\documentclass[11pt,reqno]{amsart}
\usepackage{amsthm}
\usepackage{amssymb}
\usepackage{latexsym}
\usepackage{multicol}
\usepackage{verbatim,enumerate}
\usepackage{hyperref}
\usepackage{amsmath, amscd}
\usepackage{colortbl}
\advance\textwidth by 1.2in \advance\oddsidemargin by -.6in \advance\evensidemargin by -.6in
\usepackage{tikz}
\usetikzlibrary{decorations.markings}
\tikzstyle{vertex}=[ circle, fill, draw, inner sep=0pt, minimum size=4pt,]
\tikzstyle{edge}= [thick]

\newtheorem*{cor}{Corollary}
\newtheorem*{lem}{Lemma}
\newtheorem*{prop}{Proposition}

\theoremstyle{definition}
\newtheorem*{defn}{Definition}
\theoremstyle{definition}
\newtheorem{thm}{Theorem}
\newtheorem*{thm*}{Theorem}

\newtheorem*{rem}{Remark}

\newenvironment{pf}{\proof}{\endproof}
\newcounter{cnt}
\newenvironment{enumerit}{\begin{list}{{\hfill\rm(\roman{cnt})\hfill}}{%
\settowidth{\labelwidth}{{\rm(iv)}}\leftmargin=\labelwidth%
\advance\leftmargin by \labelsep\rightmargin=0pt\usecounter{cnt}}}{\end{list}} \makeatletter
\def\mydggeometry{\makeatletter\dg@YGRID=1\dg@XGRID=20\unitlength=0.003pt\makeatother}
\makeatother \theoremstyle{remark}

\numberwithin{equation}{section}

 \DeclareMathOperator{\Ht}{ht} 
\let\bwdg\bigwedge
\def\bigwedge{{\textstyle\bwdg}}

\begin{document}

\newcommand{\thmref}[1]{Theorem~\ref{#1}}
\newcommand{\secref}[1]{Section~\ref{#1}}
\newcommand{\lemref}[1]{Lemma~\ref{#1}}
\newcommand{\propref}[1]{Proposition~\ref{#1}}
\newcommand{\corref}[1]{Corollary~\ref{#1}}
\newcommand{\remref}[1]{Remark~\ref{#1}}
\newcommand{\defref}[1]{Definition~\ref{#1}}
\newcommand{\er}[1]{(\ref{#1})}
\newcommand{\id}{\operatorname{id}}
\newcommand{\ord}{\operatorname{\emph{ord}}}
\newcommand{\sgn}{\operatorname{sgn}}
\newcommand{\wt}{\operatorname{wt}}
\newcommand{\tensor}{\otimes}
\newcommand{\from}{\leftarrow}
\newcommand{\nc}{\newcommand}
\newcommand{\rnc}{\renewcommand}
\newcommand{\dist}{\operatorname{dist}}
\newcommand{\qbinom}[2]{\genfrac[]{0pt}0{#1}{#2}}
\nc{\cal}{\mathcal} \nc{\goth}{\mathfrak} \rnc{\bold}{\mathbf}
\renewcommand{\frak}{\mathfrak}
\newcommand{\supp}{\operatorname{supp}}
\newcommand{\Irr}{\operatorname{Irr}}
\newcommand{\psym}{\mathcal{P}^+_{K,n}}
\newcommand{\psyml}{\mathcal{P}^+_{K,\lambda}}
\newcommand{\psymt}{\mathcal{P}^+_{2,\lambda}}
\renewcommand{\Bbb}{\mathbb}
\nc\bomega{{\mbox{\boldmath $\omega$}}} \nc\bpsi{{\mbox{\boldmath $\Psi$}}}
 \nc\balpha{{\mbox{\boldmath $\alpha$}}}
 \nc\bpi{{\mbox{\boldmath $\pi$}}}
  \nc\bxi{{\mbox{\boldmath $\xi$}}}
\nc\bmu{{\mbox{\boldmath $\mu$}}} \nc\bcN{{\mbox{\boldmath $\cal{N}$}}} \nc\bcm{{\mbox{\boldmath $\cal{M}$}}} \nc\blambda{{\mbox{\boldmath
$\lambda$}}}\nc\bnu{{\mbox{\boldmath $\nu$}}}

\newcommand{\Tmn}{\bold{T}_{\lambda^1, \lambda^2}^{\nu}}

\newcommand{\lie}[1]{\mathfrak{#1}}
\newcommand{\ol}[1]{\overline{#1}}
\makeatletter
\def\section{\def\@secnumfont{\mdseries}\@startsection{section}{1}%
  \z@{.7\linespacing\@plus\linespacing}{.5\linespacing}%
  {\normalfont\scshape\centering}}
\def\subsection{\def\@secnumfont{\bfseries}\@startsection{subsection}{2}%
  {\parindent}{.5\linespacing\@plus.7\linespacing}{-.5em}%
  {\normalfont\bfseries}}
\makeatother
\def\subl#1{\subsection{}\label{#1}}
 \nc{\Hom}{\operatorname{Hom}}
  \nc{\mode}{\operatorname{mod}}
\nc{\End}{\operatorname{End}} \nc{\wh}[1]{\widehat{#1}} \nc{\Ext}{\operatorname{Ext}} \nc{\ch}{\operatorname{ch}} \nc{\ev}{\operatorname{ev}}
\nc{\Ob}{\operatorname{Ob}} \nc{\soc}{\operatorname{soc}} \nc{\rad}{\operatorname{rad}} \nc{\head}{\operatorname{head}}
\def\Im{\operatorname{Im}}
\def\gr{\operatorname{gr}}
\def\mult{\operatorname{mult}}
\def\Max{\operatorname{Max}}
\def\ann{\operatorname{Ann}}
\def\sym{\operatorname{sym}}
\def\loc{\operatorname{loc}}
\def\Res{\operatorname{\br^\lambda_A}}
\def\und{\underline}
\def\Lietg{$A_k(\lie{g})(\bsigma,r)$}
\def\res{\operatorname{res}}

 \nc{\Cal}{\cal} \nc{\Xp}[1]{X^+(#1)} \nc{\Xm}[1]{X^-(#1)}
\nc{\on}{\operatorname} \nc{\Z}{{\bold Z}} \nc{\J}{{\cal J}} \nc{\C}{{\bold C}} \nc{\Q}{{\bold Q}}
\renewcommand{\P}{{\cal P}}
\nc{\N}{{\Bbb N}} \nc\boa{\bold a} \nc\bob{\bold b} \nc\boc{\bold c} \nc\bod{\bold d} \nc\boe{\bold e} \nc\bof{\bold f} \nc\bog{\bold g}
\nc\boh{\bold h} \nc\boi{\bold i} \nc\boj{\bold j} \nc\bok{\bold k} \nc\bol{\bold l} \nc\bom{\bold m} \nc\bon{\bold n} \nc\boo{\bold o}
\nc\bop{\bold p} \nc\boq{\bold q} \nc\bor{\bold r} \nc\bos{\bold s} \nc\boT{\bold t} \nc\boF{\bold F} \nc\bou{\bold u} \nc\bov{\bold v}
\nc\bow{\bold w} \nc\boz{\bold z} \nc\boy{\bold y} \nc\ba{\bold A} \nc\bb{\bold B} \nc\bc{\mathbb C} \nc\bd{\bold D} \nc\be{\bold E} \nc\bg{\bold
G} \nc\bh{\bold H} \nc\bi{\bold I} \nc\bj{\bold J} \nc\bk{\bold K} \nc\bl{\bold L} \nc\bm{\bold M} \nc\bn{\mathbb N} \nc\bo{\bold O} \nc\bp{\bold
P} \nc\bq{\bold Q} \nc\br{\bold R} \nc\bs{\bold S} \nc\bt{\bold T} \nc\bu{\bold U} \nc\bv{\bold V} \nc\bw{\bold W} \nc\bz{\mathbb Z} \nc\bx{\bold
x} \nc\KR{\bold{KR}} \nc\rk{\bold{rk}} \nc\het{\text{ht }}

\nc\toa{\tilde a} \nc\tob{\tilde b} \nc\toc{\tilde c} \nc\tod{\tilde d} \nc\toe{\tilde e} \nc\tof{\tilde f} \nc\tog{\tilde g} \nc\toh{\tilde h}
\nc\toi{\tilde i} \nc\toj{\tilde j} \nc\tok{\tilde k} \nc\tol{\tilde l} \nc\tom{\tilde m} \nc\ton{\tilde n} \nc\too{\tilde o} \nc\toq{\tilde q}
\nc\tor{\tilde r} \nc\tos{\tilde s} \nc\toT{\tilde t} \nc\tou{\tilde u} \nc\tov{\tilde v} \nc\tow{\tilde w} \nc\toz{\tilde z} \nc\woi{w_{\omega_i}}
\nc\chara{\operatorname{Char}}

\title{A Steinberg type decomposition theorem for higher level   Demazure modules}
\author[Chari, Shereen,  Venkatesh and Wand ]{Vyjayanthi Chari, Peri Shereen, R.Venkatesh  and Jeffrey Wand}\address{\noindent Department of Mathematics, University of California, Riverside, CA 92521}
\email{chari@math.ucr.edu, psheroo1@math.ucr.edu; wand@math.ucr.edu}
\address{\noindent Department of Mathematics, Tata Institute of Fundamental Research, Homi Bhabha Road, Mumbai 400005, India.}
\email{rvenkat@math.tifr.res.in}
\thanks{V.C. was partially supported by DMS-0901253 and DMS- 1303052.}
\maketitle
\begin{abstract}
We study Demazure modules which  occur  in a level $\ell$ irreducible integrable representation of an affine Lie algebra. We also assume that they  are stable under the action of the standard maximal parabolic subalgebra of the affine Lie algebra.  We prove that such  a module is isomorphic to the fusion product of  \lq\lq prime\rq\rq \  Demazure modules, where the prime factors are indexed by dominant integral weights which are either a multiple of $\ell$ or take value less than $\ell$ on all simple coroots.
Our proof depends on  a technical result  which we prove in all the classical cases and $G_2$.  Calculations with mathematica  show that this result is correct  for small values of the level. Using our result, we  show that there exist generalizations of $Q$--systems to   pairs of  weights where one of the weights is not necessarily rectangular and is of a different level.  Our results also allow us to compare the multiplicities of  an irreducible representation occuring in the tensor product of certian   pairs of irreducible representations, i.e., we establish a version of Schur positvity for such pairs of irreducible modules for a simple Lie algebra.

  \end{abstract}

\section{Introduction}

 Demazure modules associated to simple Lie algebra or more generally a  Kac--Moody Lie  algbera $\lie g$ have been studied intensively since their introduction in \cite{Dem}. These modules, which are actually modules  for a Borel subalgebra of the Lie algebra,   are indexed by a dominant integral weight $\Lambda$ and an element $w$ of the Weyl group.  In this paper we shall be concerned  with affine Lie algebras and a particular family of Demazure modules: namely those which are preserved by a maximal parabolic subalgebra containing the Borel. More precisely, let $\lie g$ be  a simple finite--dimensional complex Lie algebra and $\widehat{\lie g}$ the corresponding affine Lie algebra. Then the maximal parabolic subalgebra of interest is the current algebra $\lie g[t]$ which is the algebra of polynomial maps $\bc\to\lie g$ with the obvious pointwise bracket.   The $\lie g[t]$--stable Demazure modules are  indexed by a pair $(\ell, \lambda)$, where $\ell$ is the level of the integrable representation of $\widehat{\lie g}$ and $\lambda$ is a dominant integral weight of $\lie g$ and we denote the corresponding module by $D(\ell,\lambda)$.   In the case when $\ell=1$, these modules are interesting for a variety of reasons,  including the connection with Macdonald polynomials established in \cite{S}  for $\lie {sl}_{r+1}$ and in \cite{I} in general. 

 Our interest in these modules arise from their relationship with  the representation theory of quantum affine  algebras. This connection was originally developed in 
 \cite{Ckir}, \cite{CPweyl},  \cite{CM} where it was shown that  the classical limit of certain irreducible representations of the quantum affine algebra can be viewed as graded representations of $\lie g[t]$.  The classical limits were first related to the  $\lie g[t]$--stable  Demazure modules in level one representations of affine Lie algebras in \cite{CL} for $\lie{sl}_{r+1}$.   In that paper, the connection was also made between these modules and the fusion product defined in \cite{FL} of representations of $\lie g[t]$.
 In \cite{CM}  it was shown that a  Kirillov--Reshetikhin module  for a quantum affine algebra  is similarly related to a Demazure module when $\lie g$ is of classical type.

  In \cite{FoL1} and \cite{FoL} the authors worked with arbitrary untwisted affine Lie algebras and with particular classes of $\lie g[t]$--stable Demazure module . In the simply--laced case  for instance, they studied the modules $D(\ell,\ell\mu)$ where $\mu$ is a dominant integral weight of $\lie g$. They proved that such modules were the fusion product of the classical limit of the Kirillov--Reshetikhin modules defined in \cite{CM}. (The definition of fusion products of $\lie g[t]$--modules is recalled in Section 2 of this paper, for the moment it suffices to say that it is a procedure which defines a cyclic graded $\lie g[t]$--module structure on a tensor product of finite--dimensional  $\lie g$--modules. In particular, the underlying $\lie g$ module structure is unchanged, where we are regarding $\lie g$ as the subalgebra $\lie g[t]$ consisting  of constant maps).

A completely obvious question is:  what is the analog of the results of \cite{FoL1} and \cite{FoL}   for the module $D(\ell,\ell\mu+\lambda)$ where $\lambda$ is an arbitrary dominant integral weight.  A much less obvious,  but very interesting reason to study this question is the following:  when $\ell =2$ and  in the case of $\lie {sl}_{n+1}$,  these modules are related to the modules for the quantum affine algebra which occur in the work of Hernandez--Leclerc (see \cite{HL}).
This relationship  is  made precsie in \cite{BrCM}.

Recall that Steinberg's tensor product theorem asserts that  a simple module $L(\lambda)$ of an algebraic group over characteristic  $p$  is isomorphic to  a tensor product $L(p\lambda_1)\otimes L(\lambda_0)$ where $\lambda_0(h_i)\le p$ for all simple coroots.
Our first result establishes an analog of this replacing $p$ by $\ell$ and the tensor product by fusion product, i.e.,
$$D(\ell,\ell\mu+\lambda)\cong D(\ell,\ell\mu)*D(\ell,\lambda),$$ for all positive integers $\ell$ and dominant integral weights $\mu$ and $\lambda$ and if $\lie g$ is of classical type or $G_2$. The main obstruction to proving this result in general is a techincal propositon (Proposition \ref{key})  on the affine Weyl group which is problematic for $E_8$ and $F_4$. However, computer calculations show that this result is true for small values of $\ell$ and all $\lambda$ and $\mu$.

To continue the connection with the work of \cite{HL}, we define and study the notion of prime representations of $\lie g[t]$--modules: namely a module which is not isomorphic to a fusion product of non--trivial $\lie g[t]$--mdoules.   We prove that the modules $D(\ell,\ell\omega_i)$  where $\omega_i$ is a fundamental weight  and $D(\ell,\lambda)$ where $\lambda(h_i)\le \ell$ for all simple coroots,  are prime if $\lie g$ is simply--laced.
In fact we show that the underlying $\lie g$--module is not a tensor product of non--trivial $\lie g$--modules.
In the case when $\lie g$ is of type  type $A$ or   $D$ we show that any  Demazure module is a fusion product of prime Demazure modules.

We use our main result to study generalizations of  $Q$--systems (see \cite{HKOTY} for details, \cite{ked} for a more recent discussion and \cite{Her} ,  \cite{Nak} for the quantum analog). In the case of $\lie{sl}_{n+1}$, the $Q$--system is a classical identity of Schur functions associated to rectangular weights of a fixed height.  Equivalently, the $Q$--system is a short exact sequence $$0\to \bigotimes_{\{j: a_{i,j}=-1\}} V(\ell \omega_j)\to V(\ell\omega_i)\otimes V(\ell\omega_i)\to V(\ell+1)\omega_i)\otimes V((\ell-1)\omega_i)\to 0,$$ where $V(r\omega_i)$ is the irreducible representation of $sl_{n+1}$ with highest weight $r\omega_i$. In Theorem \ref{qsystem} of this paper,  we write down an analgous short exact sequence for the pair $V(\ell\omega_i)\otimes V(\lambda)$ for $\lambda$ staisfying the restriction that $\lambda(h_i)\le \ell$ for all simple coroots.  In  fact we show that we can replace the tensor product of $\lie{sl}_{n+1}$--modules by fusion products of $\lie{sl}_{n+1}[t]$--modules so that all the maps are completely canoncial.
It is interesting to note that the kernel is in general not a tensor or fusion product of irreducible representations of $\lie{ sl}_{n+1}$, but is a fusion product of prime Demazure modules.

{\em Acknowledgements.  Part of this work was done while the first author was vsiting the University of Paris 7 and the  Mathematics Institute in Cologne. It is a pleasure  for her  to thank  David Hernandez and  Peter Littelmann for  many discussions and for their  hospitality.  She also thanks Deniz Kuz for  helpful conversations.  The third author was partially supported by the department of mathematics at the University of California at Riverside, a fellowship from the Niels Henrik Abel Board and a postdoctoral fellowship at the Centre de Recherche mathematique (CRM)  as part of the thematic semester \lq\lq New Directions in Lie theory\rq\rq. He thanks these institutions for their support. He is also grateful to Erhard Neher for his support and encouragement  during the semester at the CRM. }

\section{Preliminaries}

\subsection{} Throughout the paper $\mathbb C$ denotes the field of complex numbers, $\mathbb Z$ the set of integers and $\bz_+$, $\bn$ the set of non--negative and positive integers respectively. Given any complex Lie algebra $\lie a$ we let $\bu(\lie a)$ be the universal enveloping algebra of $\lie a$. Also, if $t$ is any indeterminate  we let $\lie a[t]$ be the  Lie algebra of polynomial maps from $\mathbb C$ to $\lie a$ with the obvious pointwise Lie bracket: $$[x\otimes f, y\otimes g]=[x,y]\otimes fg,\ \ x,y\in\lie a,\ \ f,g\in\bc[t].$$  Let  $\ev_0: \lie a[t]\to \lie a$ be  the map of Lie algebras given by setting $t=0$.  The Lie algebra $\lie a[t]$  and its universal enveloping algebra inherit a grading from the degree grading of $\bc[t]$, thus an element $a_1\otimes t^{r_1}\cdots a_s\otimes t^{r_s}$, $a_j\in\lie a$, $r_j\in\bz_+$ for $1\le j\le s$ will have grade $r_1+\cdots+r_s$.
We shall be interested in $\bz$--graded modules for $\lie a[t]$. By this we mean a $\bz$--graded vector space $V=\oplus_{s\in\bz}V[s]$ which admits a compatible $\lie a[t]$--action, $$(\lie a\otimes t^r)V[s]\subset V[r+s].$$
A morphism of graded $\lie a[t]$--modules is just a degree zero map of $\lie a[t]$--modules. Given  $r\in\bz$ and a graded vector space $V$, we let $\tau_r^*V$ be the $r$--th graded shift of $V$. Clearly the pull--back of any $\lie a$--module $V$ by  $\ev_0$ defines the structure of a graded $\lie a[t]$--module on $V$ and we denote this module by $\ev_0^*V$.

\subsection{}  From now on  $\lie g$ will be a simple complex Lie algebra of rank $n$ and $\lie h$ a fixed Cartan subalgebra of $\lie g$. Let $R$ be the corresponding set of roots,  $\alpha_i$, $1\le i\le n$
be a set of simple roots and $R^+$ the corresponding set of positive roots and let $\theta$ be the highest root of $R^+$.  For $\alpha\in R^+$, we set $d_\alpha=1$ if $\alpha$ is long and $d_\alpha=2$ if $\alpha$ is short and $\lie g$ is not of type $G_2$. If $\lie g$ is of type $G_2$, then we set $d_\alpha=3$ if $\alpha$ is short. 
 The Weyl group $W$ of $R$ is  generated by simple reflections $\bos_i$, $1\le i\le n$ and $w_0$ denotes the unique longest element of $W$.

 Let $x^\pm_\alpha$, $\alpha\in R^+$, $h_i$, $1\le i\le n$ be a Chevalley basis for $\lie g$.  We have $$\lie g=\lie n^-\oplus\lie h\oplus \lie n^+, \ \ \lie h=\bigoplus_{i=1}^n\mathbb Ch_i,\ \ \lie n^\pm= \bigoplus_{\alpha\in R^+}\mathbb Cx^\pm_\alpha.$$
The fundamental weights  $\omega_i\in  \lie h^*$, $1\le i\le n$ are defined by setting  $\omega_i(h_j)=\delta_{i,j}$ where $\delta_{i,j}$ is the Kronecker delta symbol.    The weight lattice $P$ (resp. $P^+$) is the $\mathbb Z$--span (resp. $\mathbb Z_+$ span) of the fundamental weights. The root lattice $Q$ and  the subset $Q^+$ are defined in the obvious way using the simple roots. The co--weight lattice $L$ is the sublattice of $P$ spanned by the elements $d_i\omega_i$, $1\le i\le n$ and the co--root lattice $M$  is defined analogously. The subsets $L^+$ and $M^+$ are defined in the obvious way.
Let $\bz[P$] be the integral group ring of $P$ with basis $e(\lambda)$, $\lambda\in P$.

\subsection{}  For $\lambda\in P^+$, denote  by $V(\lambda)$ the simple finite--dimensional $\lie g$--module generated by an element $v_\lambda$ with defining relations $$\lie n^+ v_\lambda=0,\ \ h_iv_\lambda=\lambda(h_i) v_\lambda,\ \ (x_{\alpha_i}^-)^{\lambda(h_i)+1} v_\lambda=0,\ \ 1\le i\le n. $$ It is well--known that $V(\lambda)\cong V(\mu)$ iff $\lambda=\mu$ and that  any finite--dimensional $\lie g$--module is isomorphic to a direct sum of  modules $V(\lambda)$, $\lambda\in P^+$.  If $V$ is a $\lie h$--semisimple $\lie g$--module (in particular if $\dim V<\infty$),  we have $$ V=\bigoplus_{\mu\in\lie h^*}V_\mu,\ \ V_\mu=\{v\in V: hv=\mu(h)v,\ \ h\in\lie h\},$$ and we set $\wt V=\{\mu\in\lie h^*: V_\mu\ne 0\}.$  If $\dim V_\mu<\infty$ for all $\mu\in \wt V$, then we define $\ch_{\lie h} V:\lie h^*\to \bz_+$, by sending $\mu\to\dim V_\mu$. If $\wt V$ is a finite set, then $$\ch_{\lie h}V=\sum_{\mu\in\lie h^*}\dim V_\mu e(\mu)\in\bz[P].$$

\subsection{} We now define the untwisted affine Lie algebra associated to $\lie g$ and some related terminology (see \cite{Kac} for details). The affine Lie algebra  $\hat{\lie g}$ is given by
$$\hat{\lie g} = \lie g\otimes\mathbb{C}[t,t^{-1}]\oplus \mathbb{C}c\oplus\mathbb{C}d$$
where $c$ is the canonical central element, and $d$ acts as the derivation $t\frac{d}{dt}$ and commutator
$$[x\otimes t^r,  y\otimes t^s] =[x,y]\otimes t^{r+s}  + r\delta_{r,-s}(x,y)c,$$ where $(\ ,\ ): \lie g\times\lie g\to\bc$ is a symmetric nondegenerate invariant bilinear form on $\lie g$ normalized so that the square length of the long root is two.
The Cartan subalgebra of the affine Lie algebra is
$$\hat{\lie h} = \lie h \oplus \mathbb{C}c\oplus\mathbb{C}d.$$  Regard $\lie h^*$ as a subspace of $\hat{\lie h}^*$ by setting $\mu(c)=\mu(d)=0$ for all $\mu\in\lie h^*$.  Let $\delta,\Lambda_0\in\hat{\lie h}^*$ be  given by $$\delta(d)=1,\ \ \delta(\lie h\oplus\bc c)=0,\ \ \Lambda_0(c)=1,\ \ \Lambda_0(\lie h\oplus\bc d)=0.$$ Extend the non--degenerate form on $\lie h^*$ to a non--degenerate form on  $\hat{\lie h}^*$ by setting,$$
(\delta,\delta)\ \ =\ \ (\Lambda_0, \Lambda_0)=0, \ \  (\Lambda_0, \delta)=1.$$

The elements  $\alpha_i$, $0\le i\le n$ where  $\alpha_0=-\theta+\delta$ are a set of simple roots for the set of roots of $(\hat{\lie g}, \hat{\lie h})$.  Let $\hat R^+$ be the corresponding set of positive roots, $$\hat R^+=\{\alpha+r\delta:\alpha\in R, \ r\in\bn\}\ \cup R^+\ \cup\  \ \{r\delta: r\in\bn\}.$$
Set   $\hat{ \lie b}=\hat{\lie h}\bigoplus_{\alpha\in\hat R^+}\hat{\lie g}_\alpha$ and note that  $$ \lie g[t]\oplus\bc\oplus\bc d=\lie n^-\oplus\hat{\lie b},\ \qquad \lie g[t]=\lie n^-\oplus\lie h\bigoplus_{\alpha\in\hat R^+}\hat{\lie g}_\alpha.$$

\subsection{}\label{weyl} For $1\le i\le n$, set   $\Lambda_i=\omega_i+ \omega_i(h_\theta)\Lambda_0\in \hat{\lie h}^*$ . The set $\hat{P}^+$ of dominant integral affine weights is  defined to be the $\bz_+$--span of the elements $\Lambda_i+\bz\delta$,  $0\le i\le n$ and  $\hat P$  is defined similarly. The root lattice  $\hat Q$  is the $\bz$--span of the simple roots $\alpha_i$, $0\le i\le n$  and $\hat Q^+$ is defined in the obvious way. 

 The affine Weyl group $\widehat W$ acts on $\hat{\lie h}^*$ via reflections corresponding to the affine simple roots, in particular $w\delta=\delta$ for all $w\in\widehat W$.
An equivalent way to define the affine Weyl group is as follows.  The  finite Weyl group  $W$ acts on the  co--root lattice $M$ by restricting its action on $\lie h^*$ and we have $$\widehat W\cong W\ltimes t_M.$$
The extended Weyl group $\widetilde W$ is the semi--direct product of $\widehat W$ with the group of affine diagram automorphisms, denoted $\Sigma$, and $$\widetilde W\cong W\ltimes  t_{L}$$ where $L$ is the co-weight lattice.
Given $\mu\in M$ (resp. $L$) , we denote by $t_\mu$ the corresponding element of  $\widehat W$ (resp. $\widetilde W$). Then,
 \begin{equation}\label{translation} t_\mu(\lambda)=\lambda-(\lambda,\mu)\delta,\ \ \lambda\in \lie h^*\oplus\bc\delta ,\  \qquad t_\mu(\Lambda_0)=\Lambda_0 +\mu -\frac 12(\mu,\mu)\delta.\end{equation}
 Let $\bz[\hat P]$ be the integral group ring of $\hat P$ with basis $e(\Lambda)$ and let $I_\delta$ be the ideal of $\bz[\hat P]$ obtained by setting $e(\delta)=1$. Since we have identified $\lie h^*$ with a subspace of $\hat{\lie h}^*$, the group ring $\bz[P]$ is isomorphic to a subring of $\bz[\hat P]$ and the composite morphism
$$\bz[P]\hookrightarrow\bz[\hat P]\longrightarrow \bz[\hat P]/I_\delta,$$ is injective. Clearly, the action of $\widetilde W$ on $\hat P$ induces an action on $\bz[\hat P]$ and $\bz[\hat P]/I_\delta$ as well.

\subsection{}  For $\Lambda \in \widehat{P}^+$ let  $V(\Lambda)$ be the highest weight, irreducible, integrable $\hat{\lie g}$-module with highest weight $\Lambda$ and highest weight vector $v_\Lambda$. Then,
$$V(\Lambda)=\bigoplus_{\eta\in\hat Q^+} V(\Lambda)_{\Lambda-\eta},\ \   V(\Lambda)_{\Lambda-\eta}=\{v\in V(\Lambda): hv=(\Lambda-\eta)(h)v,\ \ h\in\hat{\lie h}^*\}.$$  For $w\in\widehat W$, we have $\dim V(\Lambda)_{w\Lambda}=1$ and  the corresponding Demazure module is, $$V_w(\Lambda) =\bu(\hat{\lie b})V(\Lambda)_{w\Lambda}.$$
 More generally, given,  $\sigma\in\Sigma$ and $w\in\widehat W$, set  $  V_{w\sigma}(\Lambda)= V_w(\sigma\Lambda).$ Since $V(\Lambda)_{\Lambda-\eta+r\delta}=0$ for all $r\in\bn$, it follows that 
$\dim V_{\sigma w}(\Lambda)<\infty$. In the special case when   $w\Lambda|_{\lie h}\in -P^+$,  the Demazure module $V_{w}(\Lambda)$ is $\lie g$--stable, in other words it is a finite--dimensional module for $\lie g[t]$. The action of $d$ defines a grading on $V_{w}(\Lambda)$ which is compatible with the $\bz$--grading on $\lie g[t]$. Finally, note that   for $w\in\widetilde W$, the function  $\ch_{\hat{\lie h}} V_w(\Lambda):\hat P \to\bz $ is the mapping $\Lambda'\to \dim V_w(\Lambda)_{\Lambda'}$ and is an element of $\bz[\hat P]$.

\subsection{} We recall the notion of fusion products of representations of $\lie g[t]$ introduced in \cite{FL}.
Let  $V$ be  a finite--dimensional cyclic $\lie g[t]$ module generated by an element $v$ and for $r\in\bz_+$  set  $$F^rV = \left(\bigoplus_{0\leq s \leq r} \bu(\lie g[t])[s]\right).v$$ Clearly $F^rV$ is a $\lie g$--submodule of $V$ and we have a finite $\lie g$--module filtration $$0\subset F^0V\subset F^1V\subset\cdots \subset F^pV =V,$$  for some $p\in\bz_+$. The associated graded vector space $\gr V$ acquires a graded $\lie g[t]$--module structure in a natural way and is  generated by the image of $v$ in $\gr V$.

Given  a  $\lie g[t]$ module $V $    and $z\in\bc$, let $V^z$ be the $\lie g[t]$--module with action $$(x\otimes t^r) w= (x\otimes (t+z)^r)w,\ \ x\in\lie g,\ \  r\in\bz_+, w\in V.$$ If $V_s$, $1\le s\le k$ are cyclic finite--dimensional $\lie g[t]$--modules with cyclic vectors $v_s$, $1\le s\le k$ and $z_1,\cdots, z_k$ are distinct complex numbers then, the  fusion product $V_1^{z_1}*\cdots*V_k^{z_k}$ is defined to be $\gr \bv(\boz)$, where $\bv(\boz)$ is the tensor product
$$\bv(\boz) =V_1^{z_1}\otimes\cdots\otimes V_k^{z_k}.$$
It was proved in \cite{FL} that in fact $\bv(\boz) $ is cyclic and generated by $v_1\otimes\cdots\otimes v_m$ and hence the  fusion product  is cyclic on the image   $v_1*\cdots *v_m$ of this element. Clearly the definiton of the fusion product depends on the parameters $z_s$, $1\le s\le k$. However it is conjectured in \cite{FL}  and (proved in certain cases by various people,  \cite{CL},  \cite{FF}, \cite{FL}  \cite{FoL}, \cite{Kedem} for instance)  that under suitable conditions on $V_s$ and $v_s$, the fusion product is independent of the choice of the complex numbers. For ease of notation we shall often suppress the dependence on the complex numbers and write  $V_1*\cdots*V_k$ for $V_1^{z_1}*\cdots*V_k^{z_k}$.

\subsection{}  We conclude this section with a technical result which will be needed in the proof of Theorem \ref{mainthm}. Given $w\in \widehat W$, let $\ell(w)$ be the length of a reduced expression of $w$.  Clearly $\ell(w_1w_2)\le \ell(w_1)+\ell(w_2)$ for all $w_1, w_2\in\widehat W$.  An alternative characterization of $\ell(w)$ is  \begin{equation}\label{altchlength}\ell(w)=\#\{\alpha\in\hat R^+:  w\alpha\in -\hat R^+\}.\end{equation}  It is convenient to define the length of an element in the extended Weyl group as well, by
$$ \ell(\sigma w) = \ell (w),\ \  \textrm{for } w\in\widehat{W} \textrm{ and } \sigma\in \Sigma.$$   For $w\in\widetilde W$ set  $\hat R_w^+=\{\alpha\in\hat  R^+: w\alpha\in-\hat R^+\}$. Since $\Sigma$ is the group of automorphisms  of the Dynkin diagram of $\widehat{\lie g}$ it follows that
$\ell(w)=\# \hat R^ +_w$  as well. Note also that for all $w\in \widehat{W}$ and $\sigma\in \Sigma$ we have $\ell(\sigma w\sigma^{-1})=\ell(w)$ and hence $\ell(w\sigma)=\ell(w)$. 

\begin{prop}\label{add}\begin{enumerit}
\item[(i)]
Let  $w_1,w_2\in \widetilde W$ be such that  $\hat R^+_{w_2}\subset \hat R^+_{w_1w_2}$.  Then  $\ell(w_1w_2)=\ell(w_1)+\ell(w_2).$
 \item[(ii)]For $\lambda,\mu\in P^+$ and $w\in W$ we have $\ell(t_{-\mu}t_{-\lambda}w) =\ell(t_{-\mu}) + \ell(t_{-\lambda}w).$
 \end{enumerit}
\end{prop}
\begin{pf}

 Write $w_s=\sigma_s w'_s$ for some $\sigma_s\in \Sigma$ and $w'_s\in \widehat W$  for $s=1,2$. Hence we get $$\ell(w_1w_2)= \ell(w_1'\sigma_2 w_2')=\ell ((\sigma_2^{-1}w_1'\sigma_2)w_2')\le \ell(\sigma_2^{-1}w_1'\sigma_2)+\ell(w_2')=\ell(w_1)+\ell(w_2).$$ It remains to prove the reverse inequality. For this it is enough to prove that \begin{equation}\label{ge}    \hat R^+_{w_2}\cup w_2^{-1}\hat R^+_{w_1}\subset \hat R^+_{w_1w_2} \\ \ \ \ \ \ \hat R^+_{w_2}\cap w_2^{-1}\hat R^+_{w_1}=\emptyset.\end{equation}  To prove the inclusion, we only need to show that $w_2^{-1}\hat R^+_{w_1}\subset \hat R^+$.  For  this,  note that if $\beta\in \hat R^+$, we have $$
  -\beta\in  w_2^{-1}\hat R^+_{w_1}\implies w_2\beta\in -\hat R^+_{w_1}\subset -\hat R^+\implies w_1w_2\beta\in-\hat R^+,$$  by our hypothesis.
On the other hand we also have $$  -\beta\in  w_2^{-1}\hat R_{w_1}^+\implies -w_2\beta\in \hat R^+_{w_1}\implies -w_1w_2\beta\in-\hat R^+,$$  which is clearly absurd. The second assertion in \eqref{ge} follows from    $$\alpha\in w_2^{-1}\hat R^+_{w_1}\implies w_2\alpha\in \hat R^+_{w_1}\subset \hat R^+\implies\alpha\notin\hat R^+_{w_2},$$ and part (i) of the proposition is established.

For (ii) we see that using part (i), it suffices  to prove   that if  $$\alpha+p\delta \in\hat R^+,\ \ t _{-\lambda}w(\alpha+p\delta) \in-\hat R^+\implies t_{-\mu}t_{-\lambda}w(\alpha+p\delta)\in-\hat R^+.$$
Since $\mu\in P^+$ it follows from the explicit formulae  for the translations  that $t_{-\mu}$ preserves $-(R^++\bz_+\delta)$.
 Hence it suffices to show that  $$\alpha+p\delta \in\hat R^+,\ \ t _{-\lambda}w(\alpha+p\delta) \in\hat R^-\implies t_{-\lambda}w(\alpha+p\delta)\subset -(R^++\bz_+\delta),$$ i.e., that $w\alpha\in - R^+$.
But this is again clear from the formulae  because  $\lambda\in P^+$

\end{pf}

\section{The main results}

We begin this section by giving an alternate presentation of the $\lie g$--stable Demazure modules and then state our main result in Section \ref{mainthmsection}. We then discuss applications of our results, the notion of prime modules and also  a generalization of the $Q$--systems of \cite{HKOTY}.

\subsection{}  We  introduce a family of graded modules  for $\lie g[t]$. These are indexed by a pair $(\ell,\lambda)\in \bn\times P^+$ and the corresponding module is  denoted $D(\ell,\lambda)$. For $\alpha\in R^+$,  set $s_\alpha, m_\alpha\in\bn $  by $$\lambda(h_\alpha)= d_\alpha\ell (s_\alpha-1)+ m_\alpha,\ \  0<m_\alpha\le d_\alpha\ell.$$
 Then, $D(\ell,\lambda)$ is the $\lie g[t]$--module generated by an element $w_\lambda$ with defining relations:
\begin{gather}\label{locweyl} \lie n^+[t] w_\lambda=0,\ \ ( h_i\otimes t^s)w_\lambda=\delta_{s,0}\lambda(h_i) w_\lambda,\ \ (x_{\alpha_i}^-)^{\lambda(h_i)+1} w_\lambda=0,\ \ 1\le i\le n,\\
\label{demquot1} ( x_{\alpha}^-\otimes t^{s_{\alpha}})w_{\lambda}=0,\\ \label{demquot2}  \ (x_{\alpha}^-\otimes t^{s_{\alpha}-1})^{m_{\alpha}+1}w_\lambda =0,\ \ {\rm{if}}\ \ m_\alpha<d_\alpha\ell. \end{gather}

\begin{rem}\label{locfindim} The relations in \eqref{locweyl} guarantee that the module $D(\ell,\lambda)$ is finite--dimensional (a more detailed discussion of this can be found in \cite{CPweyl}). In particular this gives,$$(x_{\alpha}^-\otimes 1)^{\lambda(h_{\alpha})+1}w_\lambda =0,$$ for all $\alpha\in R^+$.
\end{rem}

 \subsection{} The defining relations of $D(\ell,\lambda)$ are graded, it follows that $D(\ell,\lambda)$ is a graded $\lie g[t]$--module once we declare the grade of $w_\lambda$ to be zero.  Clearly for $s\in\bz$,  the graded shift $\tau_s^*D(\ell,\lambda)$ is defined  by letting  $w_\lambda$ have grade  $s$. It is elementary to check that $\ev_0^*V(\lambda)$ is the unique  irreducible graded  quotient of $D(\ell,\lambda)$ and moreover that,
\begin{equation}\label{demev} D(\ell,\lambda)\cong\ev_0^* V(\lambda), \ \ {\rm {if}}\ \
\ \ \lambda(h_\alpha)\le d_\alpha\ell,\ \ {\rm{for\ all}}\ \ \alpha\in R^+.\end{equation}
It is sometimes necessary to consider simultaneously,  the  different level Demazure modules associated to a given weight $\lambda$, in which case we shall denote the generator of $D(\ell,\lambda)$ by $w_{\lambda,\ell}$ and the integers $s_\alpha$ and $m_\alpha$ by $s_{\alpha,\ell}$ and $m_{\alpha,\ell}$ respectively.
\begin{lem}\label{difflevel} For all $(\ell,\lambda)\in \bn\times P^+$, we have, \begin{equation*}\Hom_{\lie g[t]}(D(\ell,\lambda), D(\ell+1,\lambda))=\bc.\end{equation*} Moreover any non--zero map is surjective.\end{lem}
\begin{pf}  It is clear that any element $\varphi\in \Hom_{\lie g[t]}(D(\ell,\lambda), D(\ell+1,\lambda))$ must send $w_{\lambda,\ell}$ to a scalar multiple of $w_{\lambda,\ell+1}$ and hence the space of homomorphisms is at most one--dimensional. To prove that it is exactly one we must show that $w_{\lambda,\ell+1}$ satisfies the relations of $w_{\lambda,\ell}$. Write $$\lambda(h_\alpha)=d_\alpha\ell (s_{\alpha,\ell}-1) +m_{\alpha,\ell}=d_\alpha(\ell+1)(s_{\alpha,\ell+1}-1)+m_{\alpha,\ell+1},$$ with $0<m_{\alpha,\ell}\le d_\alpha\ell$ and $0<m_{\alpha,\ell+1}\le d_\alpha(\ell+1)$ and using the uniqueness of $s_{\alpha,\ell}$ and $m_{\alpha,\ell}$, we get that either $$s_{\alpha,\ell}=s_{\alpha,\ell+1},\ \ m_{\alpha,\ell}=m_{\alpha,\ell+1}+ d_\alpha(s_{\alpha,\ell+1}-1)\ge m_{\alpha,\ell+1},$$ or $s_{\alpha,\ell}>s_{\alpha,\ell+1}$. In either case the assertion follows.
\end{pf}

\subsection{}The following result which is a combination of  \cite[Section 2.3, Corollary 1]{FoL}, \cite[Proposition 3.6]{Naoi} and \cite[Theorem 2]{CV}  explains the connection with Demazure modules.
\begin{prop}\label{altdem} Let $(\ell,\lambda)\in\bn\times P^+$ and suppose that  $w\in \widehat W$, $\sigma\in\Sigma$, $\Lambda\in\hat{P}^+$ are such that $$w\sigma\Lambda= w_0\lambda+\ell\Lambda_0.$$ Then we have an isomorphism $$D(\ell,\lambda)\cong V_w(\sigma\Lambda), $$ of $\lie g[t]$--modules and hence, for all $\mu\in P$, we have  \begin{equation}\label{gradediso}\dim D(\ell,\lambda)_\mu =\sum_{s\in\bz_{\geq0}}\dim V_w(\sigma\Lambda)_{\ell\Lambda_0+\mu+s\delta}.\end{equation}
\hfill\qedsymbol
\end{prop}

\subsection{}\label{mainthmsection}  The main result of this paper is the following theorem.
\begin{thm}\label{mainthm} Assume that $\lie g$ is of classical type or of type $G_2$.  Let $\lambda\in P^+$ and $k,\ell\in\bn$ and write $$\lambda=\ell\left(\sum_{s=1}^k\lambda^s\right)+\lambda^0,\ \ \ \ \lambda^s\in L^+,\ \ 1\le s \le k,\ \ \lambda^0\in P^+. $$  We have an isomorphism of graded $\lie g[t]$--modules, $$D(\ell,\lambda)\cong D(\ell,\lambda^0)^{z_0}*D(\ell,\ell\lambda^1)^{z_1}*\cdots *D(\ell,\ell\lambda^k)^{z_k},$$ where $z_0,\cdots, z_k$ are distinct complex numbers.  In particular, the  fusion product on the right hand side   is independent of the choice of parameters.
\end{thm}
\subsection{}   In the case when $\lambda^0=0$ the result was first proved in \cite{FoL} and a different proof was given in \cite{CV}. As in these papers, the proof of our theorem uses the theory of Demazure operators and  the following   additional key   result
 proved in Section \ref{keyproof}.

\begin{prop}\label{key} Assume that $\lie g $ is of classical type or of type $G_2$.  Let $\lambda\in P^+$ and $\ell\in\bn$  be such that $\lambda(h_i)\le d_i\ell$ for all $1\le i\le n$.  There exists $\mu\in L^+$ and $w\in W$ such that $wt_{\mu}(\ell\Lambda_0+w_0\lambda)\in\hat P^+$.\end{prop}

\begin{rem}  The restriction  on $\lie g$  in the main theorem is purely a consequence of the fact that we are able to prove Proposition \ref{key} only in the case when $\lie g$ is of classical type or of type $G_2$. Computer calculations for small values of $\ell$ show that the proposition is true for such $\ell$ for the other exceptional Lie algebras as well.  However a  proof for arbitrary $\ell$ seems difficult for $E_8$ and $F_4$.
\end{rem}

\subsection{} For the rest of the section,  we   discuss  applications of our result.  We begin by noting the following corollary of our theorem.
\begin{prop} \label{sp} Let $\ell\in\bn$, $\lambda_1\in L^+$, and $\lambda_2\in P^+$. There exists a canonical surjective map of $\lie g[t]$--modules $$D(\ell,\ell\lambda_1)*D(\ell, \lambda_2)\to D(\ell+1,(\ell+1)\mu_1)*D(\ell+1,\mu_2)\to 0$$ for all $\mu_1\in L^+$, $\mu_2\in P^+$ with $(\ell+1)\mu_1+\mu_2=\ell\lambda_1+\lambda_2$.
\end{prop}
\begin{pf}   By Theorem \ref{mainthm} we see that the proposition amounts to proving that $$\Hom_{\lie g[t]}(D(\ell,\ell\lambda_1+\lambda_2), D(\ell+1,\ell\lambda_1+\lambda_2))\ne 0.$$But this is precisely the statement of Lemma \ref{difflevel}.
\end{pf}
\begin{cor}\label{schurpos} Let  $1\le i\le n$ be  such that $\omega_i(h_\alpha)\le 1$ for all $\alpha\in R^+$.
For all $\mu,\nu\in P^+$ and $\ell\in\bn$ such that  $\ell-d_i\ge \max\{\mu(h_\alpha):\alpha\in R^+\}$ we have, $$\dim\Hom_{\lie g}(V(\nu), V(d_i(\ell+1)\omega_i)\otimes V(\mu))\le\dim\Hom_{\lie g}(V(\nu), V(
d_i\ell\omega_i)\otimes V(\mu+d_i\omega_i)).$$\end{cor}
\begin{pf} We apply the proposition by taking $\lambda_1=d_i\omega_i$ and $\mu+d_i\omega_i =\lambda_2$. The conditions on $i$ and $\mu$ imply that $(\mu+d_i\omega_i)(h_\alpha)
 \le \ell\le d_\alpha\ell$ and $\ell\omega_i(h_\alpha)\le\ell$ for all $\alpha\in R^+$. Equation \eqref{demev} now shows that all the Demazure modules involved in the proposition are actually evaluation modules and the result follows.
\end{pf}
\begin{rem} The preceding  corollary generalizes Theorem 1(ii) of \cite{CFouSag} where the case when $\mu$ is also a multiple of $\omega_i$ was proved by entirely different methods.

\end{rem}

\subsection{} We discuss now  the kernel of the map defined in Proposition \ref{sp} and whether it too, can be described in terms of Demazure modules. This question can be  related to the notion of $Q$--systems  introduced and studied  in \cite{HKOTY} for arbitrary simple Lie algebras and for a pair $(i,m)$ where $i$ is a node of the Dynkin diagram and $m\in\bn$. Analogs of this system exist for the quantum affine algebras. We refer the reader to \cite{HKOTY}, \cite{Her}, \cite{Nak} for further information. In our discussion here, we  restrict ourselves to the simply--laced case and assume that $i$ is such that $\omega_i$ is miniscule.  For  $(i,m)\in I\times \bn$  the $Q$--system is a short exact sequence of $\lie g$--modules $$0\to  \bigotimes_{j:  i\sim j} V(m\omega_j)\to V(m\omega_i)\otimes V(m\omega_i)\to V((m+1)\omega_i)\otimes V((m-1)\omega_i)\to 0, $$ where we say that $i\sim j$ if $i\ne j$ and the nodes $i$ and $j$ are connected in the Dynkin diagram.  For current algebras, it was proved in \cite{CV} that each of the modules in the short exact sequence is  a  Demazure module for $\lie g[t]$ of level $m$. In fact, a stronger statement was established:  that replacing the tensor product of $\lie g$--modules by the fusion product of $\lie g[t]$--modules gives rise to a canonical short exact sequence of $\lie g[t]$--modules.

A natural question to ask is if there is an analog of $Q$--systems associated to an arbitrary pair of dominant integral weights. In \cite{FH}, a start was made on this question where they proved that if $\ell\ge m$, then there exists a surjective map of $\lie g$--modules $$V(\ell\omega_i)\otimes V(m\omega_i) \to V((\ell+1)\omega_i)\otimes V((m-1)\omega_i)\to 0,$$ but their methods do not allow them to  determine the kernel of this map when $\ell>m$. Our next theorem, has the result of \cite{FH} as a special case (by taking $\lambda=m\omega_i$). Moreover,  the short exact sequences of $\lie g[t]$--modules are seen (by taking $\lambda=\ell\omega_i$) to be  generalizations of $Q$--systems.  It also determines the kernel of the map defined in Proposition \ref{sp} when $\lambda_1=\omega_i$.

\begin{thm}\label{genqsystem}
Assume that $\lie g$ is of type $A$ or $D$ and    let $1\le i\le n$ be such that $\omega_i(h_\alpha)\le 1$ for all $\alpha\in R^+$.
     Choose $(\ell, \lambda)\in\bn\times  P^+$  such that $$\lambda(h_i)\ge 1 ,\qquad \ell\ge\max\{\lambda(h_\alpha):\alpha\in R^+\}.$$ Let $\nu= \ell\omega_i+\lambda -\lambda(h_i)\alpha_i$ and write $\nu=\ell\nu^1+\nu^0$ for some $\nu^0\in P^+$, $\nu^1\in L^+$. There exists a canonical short exact sequence of $\lie g[t]$--modules:\begin{gather*} 0\to\tau_{\lambda(h_i)}^*\left( D(\ell, \ell\nu^1)*D(\ell, \nu^0)\right)\to D(\ell, \ell\omega_i)*D(\ell,\lambda)\\ \to D(\ell+1,(\ell+1)\omega_i)*D(\ell+1 ,\lambda-\omega_i)\to 0.\end{gather*}
\end{thm}

\subsection{}The study of graded representations of current algebras was originally motivated by the representation theory of quantum affine algebras.  In this theory it is completely natural and interesting to talk about the prime irreducible representations: namely an irreducible representation which is not isomorphic to the tensor product of non--trivial irreducible representations (see \cite{CPprime}, \cite{CMY}, \cite{HL}). An  important family of prime irreducible  representations are the Kirillov--Reshetikhin modules. Using the work of several authors (\cite{CPweyl}, \cite{Ckir},\cite{Her}, \cite{Nak}, \cite{Kedem}) together with \cite{CM} shows that the $\lie g[t]$--module $D(\ell, \ell \omega_i)$ is the \lq\lq limit\rq \rq of the corresponding Kirillov--Reshetikhin modules. Other examples of prime representations can be found in \cite{CH}, \cite{CM}, \cite{HL}.  In all these examples one actually proves that  the underlying $\lie g$--module is prime which motivates the following definition.
\begin{defn} We say that a $\lie g$--module $V$  is prime if it is not isomorphic to the tensor product of a non-trivial pair of $\lie g$--modules.
\end{defn}
 It is not hard to see that any irreducible finite--dimensional  $\lie g$--module is prime. It is also trivial to construct examples of prime representations of $\lie g$ which are reducible. For instance, in the $\lie{sl}_2$ case the direct  sum of the natural and the adjoint representation is obviously prime. In the case when $\dim V<\infty$ it is clear that any $\lie g$--module  has a prime factorization: in other words, is isomorphic to a tensor product of non--trivial prime modules. However, it is not known in general if such a decomposition is unique. The uniqueness of  a tensor product of simple $\lie g$--modules was proved fairly recently in \cite{Rajan}, \cite{VenkVis}.
  Notice that a $\lie g[t]$--module $V$ which is prime is necessarily prime with respect to the fusion product as well.

 \subsection{}  Our final  result shows that  if $\lie g$ is of type $A$ or $D$, then any Demazure module is a fusion product of prime Demazure modules.
\begin{prop}  \label{primefactor} Let $(\ell,\lambda)\in\bn\times P^+$ and let $\lie g$ be any simply--laced simple Lie algebra. The module $D(\ell,\lambda)$ is prime if $\lambda=\ell\omega_i$ for some $i\in I$ or $\lambda(h_i)<\ell$ for all $1\le i\le n$.   More generally, if  $\lambda=\lambda^0+\sum_{i\in }m_i\ell\omega_i$ where $0\le \lambda^0(h_i)<\ell$ for all $1\le i\le n$,   and $\lie g$ is of type $A$ or $D$, then the isomorphism  \begin{equation}\label{primedem}D(\ell,\lambda)\cong_{\lie g[t]} D(\ell, \ell\omega_1)^{*m_1}*\cdots *D(\ell, \ell\omega_n)^{*m_n}*D(\ell,\lambda^0),\end{equation} is a prime factorization of $D(\ell,\lambda)$.

\end{prop}
\begin{rem}In \cite{BrCM} the relationship of these prime Demazure modules  to prime representations of  quantum affine algebras is studied.\end{rem}

 \section{Proof of Theorem \ref{mainthm}}\label{mainthmpf}
In this section we shall assume Proposition \ref{key} and prove Theorem \ref{mainthm}. As in \cite{FoL1} and \cite{V}, the proof uses the Demazure operators and the Demazure character formula in a crucial way. We recollect these concepts briefly and refer the interested reader to  \cite{Dem}, \cite{FoL1}, \cite{Kumar} and \cite{Mathieu} for a more detailed discussion.

\subsection{} There are two main ingredients in the proof of the Theorem. The first is the following proposition which was proved in \cite{V} but we include a  very brief  sketch of the proof for the reader's convenience.

\begin{prop}\label{quotient}   Let $(\ell,\lambda)\in \mathbb{N}\times P^+$. Let  $(p_j, \mu_j )\in \bn\times  L^+$ for $1\leq j\leq m$ be  such that there exists $\mu\in P^+$ with
$$\ell\mu = p_1\mu_1+\cdots+p_m\mu_m, \ \ \mu(h_{\alpha})\geq\sum_{j=1}^m\mu_j(h_{\alpha}),\ \ {\rm{for\ all }}\ \ \  \alpha\in R^+.$$
There exists a non-zero surjective map of graded $\lie g[t]$-modules,
$$D(\ell,\ell\mu+\lambda)\longrightarrow D(p_1,p_1\mu_1) * \cdots * D(p_m,p_m\mu_m) * D(\ell,\lambda)\rightarrow0.$$ 
\end{prop}

\begin{proof}  
For $\alpha\in R^+$, and $1\le j\le m$, write $$\lambda(h_\alpha)= d_\alpha\ell(r_\alpha-1) +m_\alpha,\ \ \ 0< m_\alpha\le d_\alpha \ell,\qquad \mu(h_\alpha)=d_\alpha s_\alpha,\ \ \mu_j(h_\alpha)= d_\alpha s^j_{\alpha}. $$
 For $1\le j\le m$ set $v_j=w_{p_j\mu_j}$ and recall that $$(x_\alpha^-\otimes t^{s^j_{\alpha}})v_j=0,\quad (x^-_\alpha\otimes t^{r_\alpha})w_\lambda=0,\ \   (x^-_\alpha\otimes t^{r_\alpha-1})^{m_\alpha+1}w_\lambda=0.$$
Let $\bow$ be  the image of $v_1\otimes\cdots\otimes v_m\otimes w_\lambda$  in $ D(p_1,p_1\mu_1) * \cdots * D(p_m,p_m\mu_m) * D(\ell,\lambda)$.
  The proposition follows if we show that for $\alpha\in R^+$,
 \begin{gather}\label{demrel} (x_\alpha^-\otimes t^{s_\alpha+r_\alpha})\bow=0,\qquad  {\rm{and}}\ \
 (x^-_{\alpha}\otimes t^{s_{\alpha}+r_\alpha-1})^{m_{\alpha}+1}\bow =0 , \ \ {\rm{if}}\ \  m_\alpha<d_\alpha\ell.\end{gather}
Set  $b_\alpha=s_\alpha-\sum_j s_{\alpha}^j$  and note that our assumptions imply that $b_\alpha\ge 0$. For $z_1,\cdots , z_{m+1}$  be  the distinct complex numbers which define the fusion product. This means that in the corresponding tensor product,  we have
\begin{gather*}
(x^-_\alpha\otimes t^{b_\alpha}(t-z_1)^{s_{\alpha}^1}\cdots(t-z_m)^{s_{\alpha}^m}(t-z_{m+1})^{r_\alpha})(v_1\otimes\cdots \otimes v_m\otimes v_{m+1} )\\
=\sum_{j=1}^{m+1}\left(v_1\otimes\cdots\otimes(x^-_\alpha\otimes  t^{s^j_{\alpha}} g_j(t) v_j)\otimes\cdots\otimes v_{m+1}\right)=0,
\end{gather*}  where $v_{m+1}=w_\lambda$ and $g_j(t)=\prod_{r\ne j}(t-z_r+z_j)^{s_{\alpha}^r}$. It is now immediate that $ (x_\alpha^-\otimes t^{s_\alpha+r_\alpha})\bow=0$. The proof of the second equality in \eqref{demrel} is identical and we omit the details.

\end{proof}

\subsection{} The second result that we  need  is the following.
\begin{prop}\label{demdimgen}  For  $(\ell,\lambda)\in \mathbb{N}\times P^+$ and $(\ell,\mu)\in \bn\times  L^+$, we have,
$$\dim D(\ell, \ell\mu+\lambda)=\dim D(\ell,\lambda)  \dim D(\ell, \ell\mu).$$
\end{prop}

Assuming Proposition \ref{demdimgen} the proof of Theorem \ref{mainthm} is completed as follows. It was proved in \cite{FoL} that if $\mu_s\in L^+$ for $1\le s\le m$, then $$\dim D(\ell, \ell\mu)=\prod_{s=1}^m\dim D(\ell, \ell\mu_s),$$ where $\mu=\sum_{s=1}^m\mu_s$.  Using Proposition \ref{demdimgen}, we get
  $$\dim D(\ell,\ell\mu+\lambda)=\dim\left( D(\ell, \ell\mu_1) * \cdots * D(\ell,\ell\mu_m) * D(\ell,\lambda)\right).$$  Taking $p_1=\cdots p_m=\ell$ in Proposition \ref{quotient} now establishes Theorem \ref{mainthm}.

\subsection{} The rest of the section is devoted to the proof of Proposition \ref{demdimgen}.
  Recall from Section \ref{weyl} that the composite map $$\bz[P]\hookrightarrow\bz[\hat P]\longrightarrow \bz[\hat P]/I_\delta,$$ is injective. Given two elements $\chi,\chi'$ of $\bz[\hat P]$, we write $\chi\equiv\chi'$ if they have the same image in $\bz[\hat P]/I_\delta$.

\begin{lem}\label{chardl} Let $w\in \widehat W$, $\sigma\in\Sigma$, $\Lambda\in\hat{P}^+$ and $(\ell,\lambda)\in \bn\times P^+$ be such that $w\sigma\Lambda= w_0\lambda+\ell\Lambda_0$. Then $\ch_{\lie h}D(\ell,\lambda)=\sum_{\mu\in P}\dim D(\ell,\lambda)_\mu e(\mu) \in\bz[P]$ is invariant under the action of $W$ on $P$ and we have 
 $$\ch_{\hat{\lie h}} V_w(\sigma\Lambda)\equiv e(\ell\Lambda_0)\ch_{\lie h}D(\ell,\lambda).$$
\end{lem}

\begin{pf}
 The fact that $\ch_{\lie h}D(\ell,\lambda)$ is $W$--invariant is immediate since $D(\ell,\lambda)$ is a finite--dimensional $\lie g$--module. Recall that,
\begin{gather*}
\ch_{\hat{\lie h}}V_w(\sigma\Lambda)= \sum_{\Lambda' \in \hat{P}} \dim(V_w(\sigma\Lambda)_{\Lambda'})e(\Lambda').\end{gather*}
Since $\Lambda(c)=\ell$, we may assume that the sum is over elements of $\hat P$ of the form
 $\ell\Lambda_0+\mu+s\delta$ for $\mu\in P$ and $s\in\bz_{\geq0}$.  Going mod $I_\delta$, we get that $$
\ch_{\hat{\lie h}}V_w(\sigma\Lambda)\equiv e(\ell\Lambda_0)
 \sum_{\mu \in P} \left(\sum_{s\in \bz_{\geq0}}\dim V_w(\sigma\Lambda)_{\ell\Lambda_0+\mu+s\delta}\right)e(\mu) =e(\ell\Lambda_0)\ch_{\lie h}D(\ell,\lambda),$$ where the last equality follows from \eqref{gradediso}.

\end{pf}\

\subsection{}  For $0\le i\le n$, the Demazure operator $D_i:\bz[\hat P]\to\bz[\hat P]$ is defined  by,
$$D_i(e(\Lambda))=\frac{e(\Lambda)-e(\bos_i(\Lambda)-\alpha_i)}{1-e(-\alpha_i)}.$$
 Here for $1\le i\le n$ we identify the  generator $\bos_i$ of $W$ with the element $(\bos_i, 0)$ of $\widehat W$ and $\bos_0=(s_\theta, t_{\theta})$.
Given a reduced expression $w= \bos_{i_1}\cdots\bos_{i_r}$ for an element $w\in\widehat W$,  set $D_w = D_{{i_1}}\cdots D_{{i_r}},$  and note that  $D_w$ is independent of the choice of reduced expression for $w$ (see \cite{Kumar1}, Corollary 8.2.10). For  $\sigma\in\Sigma$,
 and $w\in\widehat W$, set $D_{w\sigma }(e(\Lambda))= D_{w}(e(\sigma(\Lambda))$.   Since  $D_i(e(\delta))=e(\delta)$, it follows that for all $w\in\widetilde W$, the operator  $D_w$  descends to $\bz[\hat P]/I_\delta$.

The following result is proved in \cite[Lemma 6, Lemma 7, Section 3]{FoL1}.
\begin{lem}\label{winv} Let $\chi\in\bz[P]$ be a $W$--invariant element of $\bz[P]$.  Then $D_w(\chi)\equiv \chi $ for all $w\in\widetilde W$. Moreover, for all $\Lambda\in\hat  P$, we have $$D_w(e(\Lambda)\chi)\equiv \chi D_w(e(\Lambda)).$$
\hfill\qedsymbol
\end{lem}
 Along with Lemma \ref{chardl}, we get \begin{equation}\label{demwinv} D_w(e(\ell\Lambda_0)\ch_{\lie h} D(\ell,\lambda))\equiv D_w(e(\ell\Lambda_0))\ch_{\lie h} D(\ell,\lambda),\end{equation} for all $(\ell,\lambda)\in \bn\times P^+$ and $w\in\widetilde W$.

\subsection{} The following result  may be found in \cite[Theorem 3.5]{Kumar} and \cite[Theorem 8.2.9]{Kumar1}.
\begin{thm}\label{demcharacterformula}
For  $w\in\widehat W$, $\sigma\in\Sigma$, and $\Lambda\in\widehat{P}^+$ we have
$$ \ch_{\hat{\lie h}}V_w(\sigma{\Lambda}) = D_{w\sigma}(e(\Lambda)).$$\hfill\qedsymbol
\end{thm}

Lemma  \ref{chardl} and Theorem \ref{demcharacterformula} now gives,\begin{equation}\label{finally} D_{w\sigma}(e(\Lambda)) \equiv e(\ell\Lambda_0)\ch_{\lie h} D(\ell,\lambda),\end{equation} for all $\sigma\in \Sigma$ and  $w\in\widehat W$  such that $w\sigma\Lambda= w_0\lambda+\ell\Lambda_0$.

\subsection{} The next result makes crucial use of Proposition \ref{key}.\begin{lem}\label{keyuse}  Let $\ell\in\bn$ and $\lambda\in P^+$  be such that   $\lambda=\ell\lambda_1+\lambda_2$ where $\lambda_1\in L^+$ and $\lambda_2\in P^+$ satisfies $\lambda_2(h_i)\le d_i\ell$ for all $1\le i\le n$.
Then,  $$\ch_{\lie h}D(\ell, \lambda)=\ch_{\lie h}D(\ell,\ell\lambda_1)\ch_{\lie h} D(\ell,\lambda_2).$$
\end{lem}
\begin{pf} By Proposition \ref{key}  we can choose   $\nu\in L^+$ and $w\in W$ such that $$\Lambda=w^{-1}t_\nu(\ell\Lambda_0+w_0\lambda_2)\in \hat P^+.$$ Since $t_{w_0\lambda_1}t_{-\nu}w(\Lambda)=\ell\Lambda_0+w_0\lambda+m\delta$ for some $m\in\bz$, it follows from \eqref{finally} that $$e(\ell\Lambda_0){\ch}_{\lie h}D(\ell,\lambda) \equiv D_{t_{w_0\lambda_1}t_{-\nu}w}(e({\Lambda})).$$
 Proposition \ref{add} gives $$\ell(t_{w_0\lambda_1}t_{-\nu}w) =\ell(t_{w_0{\lambda_1}})+\ell(t_{-\nu}w),$$  and hence using the  properties of Demazure operators we get,
\begin{gather*}
  D_{t_{w_0\lambda_1}t_{-\nu}w}(e({\Lambda})) =D_{t_{w_0\lambda_1}}D_{t_{-\nu}w}(e({\Lambda})). \end{gather*} Using  \eqref{finally} we get $$  D_{t_{w_0\lambda_1}}D_{t_{-\nu}w}(e({\Lambda})) \equiv D_{t_{w_0\lambda_1}}(e({\ell\Lambda_0})\ch_{\lie h} D(\ell,\lambda_2)).$$  Using \eqref{demwinv} and a further application of \eqref{finally} gives, 
\begin{gather*}
 D_{t_{w_0\lambda_1}}(e({\ell\Lambda_0})\ch_{\lie h} D(\ell,\lambda_2))
 \equiv D_{t_{w_0\lambda_1}}(e({\ell\Lambda_0}))\ch_{\lie h} D(\ell,\lambda_2)\\
  \equiv e({\ell\Lambda_0})\ch_{\lie h} D(\ell,\ell\lambda_1)\ch_{\lie h} D(\ell,\lambda_2).\end{gather*} Hence we get $${\ch}_{\lie h}D(\ell,\lambda) \equiv \ch_{\lie h} D(\ell,\ell\lambda_1)\ch_{\lie h} D(\ell,\lambda_2)$$ and   the Lemma follows since the map $\bz[P]\to \bz[\hat P]/I_\delta$ is injective.
\end{pf}

\subsection{}
 Proposition \ref{demdimgen} follows if we prove that for all $\lambda\in P^+$  and $\mu\in L^+$, we have $$D(\ell,\ell\mu+\lambda)\cong_{\lie g} D(\ell,\ell\mu)\otimes D(\ell,\lambda).$$ Since finite--dimensional $\lie g$--modules are determined by their characters, it suffices to prove that $$\ch_{\lie h}D(\ell,\ell\mu+\lambda) = \ch_{\lie h}D(\ell,\ell\mu)\ch_{\lie h} D(\ell,\lambda).$$
Write $\lambda=\ell\lambda_1+\lambda_2$ where $\lambda_1\in L^+$ and $\lambda_2\in P^+$ satisfies $\lambda_2(h_i)<d_i\ell$ for all $1\le i\le n$. By Lemma \ref{keyuse}, we get \begin{gather*} D(\ell,\ell\mu+\lambda)\cong_{\lie g}  D(\ell,\ell\mu+\ell\lambda_1)\otimes D(\ell,\lambda_2)\\ \cong_{\lie g} D(\ell,\ell\mu)\otimes D(\ell,\ell\lambda_1)\otimes D(\ell,\lambda_2)\\ \cong_{\lie g} D(\ell,\ell\mu)\otimes D(\ell,\lambda),\end{gather*} where the second  and the the third isomorphisms are a further application of Lemma \ref{keyuse}.

\section{Proof of Theorem \ref{genqsystem}}\label{qsystem}
Throughout this section $\lie g$ is simply--laced and $i\in I$ is such that  $\omega_i(h_\alpha)\le 1$ for all $\alpha\in R^+$. In particular, this means that the multiplicity of $\alpha_i$ in any positive root is at most one. We also fix  $(\ell,\lambda)\in \bn\times P^+$ with  $\lambda(h_\alpha)\le \ell$ for all $\alpha\in R^+$, and write  $$(\ell\omega_i+\lambda)(h_{\alpha})=\ell (s_{\alpha,\ell}-1)+ m_{\alpha,\ell},\ \  0<m_{\alpha, \ell}\le \ell\ \ \alpha\in R^+.$$ For $\alpha=\sum_{j=1}^nr_j\alpha_j$, set $$\supp\alpha=\{j\in I: r_j>0\}.$$

\subsection{} 

\begin{prop}\label{refined}  The defining  relation, \eqref{demquot2}, of $D(\ell,\ell\omega_i+\lambda)$  is a consequence of  \eqref{locweyl}, \eqref{demquot1} and the single additional relation,  \begin{equation}\label{demquot3}  (x_{\alpha_i}^-\otimes t)^{\lambda{(h_i)}+1}w_{\ell\omega_i+\lambda }=0.\end{equation} \end{prop}

\begin{pf} A  simple calculation shows that either $s_{\alpha_i,\ell}=1$ and $\lambda(h_i)=0$ or  $s_{\alpha_i,\ell}=2$ and $m_{\alpha_i,\ell }=\lambda(h_i)$. In the first case, the relation \eqref{demquot1} and in the second case the relation \eqref{demquot2} shows that  the relation \eqref{demquot3} does hold in $D(\ell,{\ell\omega_i+\lambda })$.

 If $\omega_i(h_\alpha)=0$, then $s_{\alpha,\ell}=1$ and $m_{\alpha,\ell}=(\ell\omega_i+\lambda )(h_\alpha)=\lambda(h_\alpha)$.  For such $\alpha$ the relation \eqref{demquot2} is $(x^-_\alpha\otimes 1)^{(\ell\omega_i+\lambda )(h_\alpha)+1}w_{\ell\omega_i+\lambda}=0$ which is
the content of Remark \ref{locfindim}.
 It remains to consider the case when $\omega_i(h_\alpha)=1$ and $\alpha\ne\alpha_i$.  If $\lambda(h_\alpha)=0$, then $m_{\alpha,\ell}=\ell$ and there is nothing to check.  Otherwise, $\lambda(h_\alpha)>0$ and 
$s_{\alpha,\ell}=2$,  $m_{\alpha,\ell}=\lambda(h_\alpha)$.
We proceed by induction on $\Ht\alpha$ with induction obviously beginning with $\alpha=\alpha_i$.  Writing $\alpha=\beta+\gamma$ for some positive roots $\beta$ and $\gamma$, we assume without loss of generality that $i\notin\supp\gamma$. Since $\alpha(h_\alpha)=2$, and we are in the simply laced case, it follows that $$(\alpha, \beta)= (\alpha,\gamma) =1,\ \ \beta-\gamma \notin R,\ \  \beta+\alpha\notin R.$$
 By the inductive hypotheses we have \begin{equation}\label {inductivehpo}(x^-_\beta\otimes t)^{\lambda(h_\beta)+1}w_{{\ell\omega_i+\lambda }}=0.\end{equation} Suppose for a contradiction that $$ (x^-_\alpha\otimes t)^{\lambda(h_\alpha)+1}w_{\ell\omega_i+\lambda } \ne 0.$$  Since $$(\ell\omega_i+\lambda-(\lambda(h_\alpha)+1)\alpha)(h_\gamma)= (\lambda-(\lambda(h_\alpha)+1)\alpha)(h_\gamma)=-\lambda(h_\beta)-1<0,$$  we get by applying the representation theory of $\lie{sl}_2$ to $x^\pm_\gamma , h_\gamma$ that $$(x^+_\gamma)^{\lambda(h_\beta)+1}(x^-_\alpha\otimes t)^{\lambda(h_\alpha)+1}w_{\ell\omega_i+\lambda } \ne 0.$$ Since $$[x_\gamma^+, x^-_\alpha]= Ax^-_\beta,\ \  [x^-_\alpha,x^-_\beta]=0\ \  [x_\beta^-, x_\gamma^+]=0,$$ for some non--zero constant $A$,  it follows by using the first two relations in  \eqref{locweyl} that  $$(x^-_\alpha\otimes t)^{\lambda(h_\gamma)}(x^-_\beta\otimes t)^{\lambda(h_\beta)+1}w_{\ell\omega_i+\lambda } \ne 0,$$which contradicts \eqref{inductivehpo} and  completes the proof.

 \end{pf}

\subsection{} We now prove,
\begin{lem} Suppose that $\lambda(h_i)>0$ and $(\ell,\lambda)\in \bn\times P^+$.  There exists a surjective map of graded $\lie g[t]$--modules $$\pi: D(\ell, \ell\omega_i+\lambda)\to D(\ell+1, \ell\omega_i+\lambda)\to  0,$$ with $$\ker\pi=\bu(\lie g[t])(x_{\alpha_i}^-\otimes t)^{\lambda(h_i)}w_{ \ell\omega_i+\lambda }.$$
\end{lem}
\begin{pf} The existence of a non--zero map $\pi: D(\ell, \ell\omega_i+\lambda)\to D(\ell+1, \ell\omega_i+\lambda)\to  0$, is guaranteed by Lemma \ref{difflevel}.  Since $\ell\omega_i+\lambda=(\ell+1)\omega_i+(\lambda-\omega_i)$ and $\lambda-\omega_i\in P^+$, it follows that Proposition \ref{refined} applies to both  $D(\ell, \ell\omega_i+\lambda)$ and to $ D(\ell+1, \ell\omega_i+\lambda)$. In particular, \eqref{demquot3} shows that$$(x_{\alpha_i}^-\otimes t)^{\lambda(h_i)}w_{ \ell\omega_i+\lambda }\in\ker\pi.$$   To prove that it generates the kernel, notice first that $w_{ \ell\omega_i+\lambda }$ and $\pi(w_{ \ell\omega_i+\lambda })$ both satisfy all the relations in \eqref{locweyl}. The Lemma follows if we  prove that $(x^-_\alpha\otimes t^{s_{\alpha,\ell}})w_{ \ell\omega_i+\lambda }$  is in the $\lie g[t]$---submodule of $D(\ell, { \ell\omega_i+\lambda })$ generated by $(x^-_{\alpha_i}\otimes t)^{\lambda(h_i)}w_{ \ell\omega_i+\lambda }$, where  $${( \ell\omega_i+\lambda )}(h_\alpha)=\ell(s_{\alpha,\ell}-1)+m_{\alpha,\ell}= (\ell+1)(s_{\alpha,\ell+1}-1)+m_{\alpha,\ell+1}.$$   If $i\notin\supp\alpha$, then $s_{\alpha,\ell}=s_{\alpha,\ell+1}=1$ and so  $(x^-_\alpha\otimes t^{s_{\alpha,\ell+1}})w_{\ell\omega_i+\lambda}=0$ and there is nothing to prove.  If $i\in\supp\alpha$  and $\lambda(h_\alpha)>1$ then $(\lambda-\omega_i)(h_\alpha)>0$ and so  $s_{\alpha,\ell}=s_{\alpha,\ell+1}=2$ and we are done. It remains to consider the case when $\lambda(h_\alpha)=\omega_i(h_\alpha)=1$. In this case \begin{equation}\label{step}s_{\alpha,\ell}=2, \ \ m_{\alpha,\ell}=1, \ \ s_{\alpha,\ell+1}=1, \ \ m_{\alpha,\ell+1}=\ell+1\end{equation} and the only thing to check is that  $(x^-_\alpha\otimes t)w_{ \ell\omega_i+\lambda } $ is in the $\lie g[t]$--submodule of $D(\ell, { \ell\omega_i+\lambda })$ generated by $(x^-_{\alpha_i}\otimes t)w_{ \ell\omega_i+\lambda }$. For this we proceed by induction on $\Ht\alpha$. If $\Ht\alpha=1$, then $\alpha=\alpha_i$ and hence induction begins.
Write $\alpha=\beta+\gamma$ with $i\in\supp\beta$ in which case $i\notin\supp\gamma$. Notice that  $$\lambda(h_\alpha)=1\implies  \lambda(h_\beta)=1,\ \ (\ell\omega_i+\lambda) (h_\gamma)=0.$$ Hence using the induction hypothesis for $\beta$ and the third equality in \eqref{locweyl} for $\gamma$, we get
$$(x^-_\alpha\otimes  t)w_{\ell\omega_i+\lambda} =x^-_\gamma(x^-_\beta \otimes t)w_{ \ell\omega_i+\lambda } \in  \bu(\lie g[t])(x^-_{\alpha_i}\otimes t)w_{{ \ell\omega_i+\lambda }}. $$   This completes the proof of the Lemma.

\end{pf}

\subsection{} The following Lemma now clearly completes the proof of Theorem \ref{genqsystem}.
\begin{lem} Suppose that $\lambda(h_i)>0$ and $(\ell,\lambda)\in \bn\times P^+$ and let  $\mu= \ell\omega_i+\lambda-\lambda(h_i)\alpha_i$.
The assignment $w_\mu\to (x^-_i\otimes t)^{\lambda(h_i)}w_{\lambda + \ell\omega_i}$ defines an injective map of $\lie g[t]$--modules  $$\iota: \tau_{\lambda(h_i)}^* D(\ell,\mu)\to D(\ell,{\lambda + \ell\omega_i}).$$
\end{lem}
\begin{pf}  Choose $\Lambda\in\hat P^+$ such that $w\Lambda=w_0(\ell\omega_i+\lambda)+\ell\Lambda_0$ for some $w\in\widehat W$.  Then,  $$D(\ell,\ell\omega_i+\lambda)\cong_{\lie g[t]} V_w(\Lambda).$$  The element $w_{\ell\omega_i+\lambda}$ maps to a non--zero  element
$v_{w_0w\Lambda}\in (V_w(\Lambda))_{w_0w\Lambda}.$ Since
 $$(w_0w\Lambda, -\alpha_i+\delta)= (\ell\omega_i+\lambda+\ell\Lambda_0,-\alpha_i+\delta)=-(\lambda,\alpha_i)<0,$$ it follows from the representation theory of the $\lie{sl}_2$ associated to the root $-\alpha_i+\delta$ that $$0\ne (x_i^-\otimes t)^{\lambda(h_i)}v_{w_0w\Lambda}\in V_w(\Lambda)_{\bos_{\alpha_i-\delta}w_0w\Lambda},$$ where $\bos_{\alpha_i-\delta}$ is the reflection in $\widehat W$ corresponding to the root $\alpha_i-\delta$. In particular,$$ (x_i^-\otimes t)^{\lambda(h_i)}w_{\ell\omega_i+\lambda}\ne 0.$$
 Since $V_w(\Lambda)$ is a $\lie g$--stable Demazure module, it follows that the $\lie g$--module through $(x_i^-\otimes t)^{\lambda(h_i)}v_{w_0w\Lambda}$ is contained in it and hence we get that $$V(\Lambda)_{w_0\bos_{\alpha_i-\delta}w_0w\Lambda}\subset V_w(\Lambda).$$ This means that we have an inclusion of Demazure modules  $V_{w_0\bos_{\alpha_i-\delta}w_0w\Lambda}(\Lambda)\hookrightarrow V_w(\Lambda)$. A straightforward calculation now shows that $$V_{w_0\bos_{\alpha_i-\delta}w_0w\Lambda}(\Lambda)\cong_{\lie g[t]} \tau_{\lambda(h_i)}^*D(\ell, \mu)$$
 which completes the proof.

\end{pf}

\section{Proof of Proposition \ref{primefactor}}\label{prime}
To prove Proposition \ref{primefactor} we must show that if $(\ell,\lambda)\in \bn\times P^+$ is such that $\lambda(h_i)\le \ell$, then $D(\ell,\lambda)$ is prime. We shall prove this in the rest of the section {\em assuming that $\lie g$ is simply--laced, including the algebras of type $E$.}

\subsection{} The first step in proving Proposition \ref{primefactor} is,
\begin{lem}\label{step1} Let $V$ be a finite--dimensional  $\lie g$--module such that: $$\dim V_\lambda=1, \qquad \wt V\subset\lambda-Q^+.$$ Suppose that $V\cong V_1\otimes V_2$, where $V_j$, $j=1,2$ are non--trivial finite--dimensional $\lie g$--modules. There exists a  unique pair of non--zero elements $\mu_j\in \wt V_j\cap P^+$ such that $$\mu_1+\mu_2=\lambda,\qquad \dim\Hom_{\lie g}(V(\mu_j), V_j)=1,$$ and an injective map $V(\mu_1)\otimes V(\mu_2)\to V$ of $\lie g$--modules.
\end{lem}
\begin{pf} The existence of $\mu_j\in\wt V_j$, $j=1,2$, such that $\mu_1+\mu_2=\lambda$ is a consequence of the fact that $\dim V_\lambda>0$ while the uniqueness of these elements is a consequence of the fact that $\dim V_\lambda=1$. Notice that this also proves that $\dim (V_j)_{\mu_j}=1$ for $j=1,2$. Since $\wt V\subset \lambda-Q^+$ we get  $\wt V_j\subset\mu_j-Q^+$ and hence $$\dim\Hom_{\lie g}(V(\mu_j), V_j)=1, \ \  j=1,2.$$ If $\mu_1=0$ then the argument proves that $V_1$ is the one--dimensional trivial representation of $\lie g$ contradicting our assumptions. This completes the proof of the Lemma.
\end{pf}

\subsection{}\label{convention}  For the rest of the section we fix $(\ell,\lambda)\in \bn\times P^+$ and an isomorphism  $$D(\ell,\lambda)\cong_{\lie g} V_1\otimes V_2,$$ 
for some finite--dimensional $\lie g$--modules $V_1$ and $V_2$.
Since $D(\ell,\lambda)$ satisfies the conditions of Lemma \ref{step1} we choose $\mu_1$ and $\mu_2$ as in Lemma \ref{step1} and
 Proposition \ref{primefactor} follows if we prove  that  either $\mu_1=0$ or $\mu_2=0$.

\subsection{} We need some additional notation.  Given any connected subset $J\subset\{1,\cdots, n\}$ of the Dynkin diagram of $\lie g$,  set
 $$R^+_J=R^+\cap\sum_{j\in J}\bz\alpha_j,\ \ \ P^+_J= P^+\cap\sum_{j\in J}\bz\omega_j, \ \ \ Q^+_J= Q^+\cap\sum_{j\in J}\bz\alpha_j.$$

Let $\lie g_J$ be the subalgebra of $\lie g$ generated by the elements $x_i^\pm$, $i\in  J$ and let  $\lie n^\pm_J$, $\lie h_J$ be  defined in the obvious way.  Then $R^+_J$ is the set of positive roots of $\lie g_J$ with respect to $\lie h_J$ and $P_J$ and $Q_J$ are  the corresponding weight and root lattice  respectively.
Finally, we  regard the algebra $\lie g_J[t]$ as a subalgebra of $\lie g[t]$ in the natural way.

   Given $\mu\in P^+$ set $$V_J(\mu)=\bu(\lie g_J)v_\mu\subset V(\mu) ,\ \qquad   \ D_J(\ell,\mu)=\bu(\lie g_J[t]))w_\mu\subset D(\ell, \mu). $$  Then $V_J(\mu)$ is the irreducible $\lie g_J$--module with highest weight   $\mu_J$  which is the restriction of $\mu$ to $\lie h_J$. The module $D_J(\ell,\mu)$ is a quotient of the Demazure module for $\lie g_J[t]$ associated to the pair $(\ell,\mu_J)$.

The following is elementary and will be used repeatedly.
\begin{lem}\label{elementaryJ} \begin{enumerit} \item [(i)] Suppose that $\mu,\mu'\in P^+$ and $\eta\in Q_J^+$ is such that $\nu= \mu+\mu'-\eta\in P^+$. Then $$\Hom_{\lie g_J}(V_J(\nu), V_J(\mu')\otimes V_J(\mu))\cong \Hom_{\lie g}(V(\nu), V(\mu')\otimes V(\mu)).$$
\item[(ii)]  Suppose that $\mu,\nu\in P^+$ are such that $\mu-\nu\in Q_J^+$. Then, \begin{equation}\label{demJ}\dim\Hom_{\lie g_J}( V_J(\nu), D_J(\ell, \mu))=\dim\Hom_{\lie g}(V(\nu), D(\ell,\mu)).\end{equation}
\end{enumerit}
\hfill\qedsymbol \end{lem}

\subsection{} For $\mu\in P^+$, set $\supp\mu=\{i\in I: \mu(h_i)>0\}.$
\begin{lem}\label{suppintempty} Let $(\ell,\lambda)\in\bn\times P^+$ with $\lambda(h_i)\le \ell$ for all $ 1\le i\le n$. With the notation of Section \ref{convention},  we have $$\supp\mu_1\cap\supp\mu_2=\emptyset.$$ In particular, if $\lambda =m\omega_i$ for some $0\le m\le \ell$ and we are in the simply laced case, then $D(\ell,\lambda)$ is prime.\end{lem}
\begin{pf}
Suppose for a contradiction   that $i\in\supp\mu_1\cap \supp\mu_2$ for some $1\le i\le n$ and set $ J=\{i\}$. Then  $\lie g_J\cong\lie{sl}_2$ and hence using the Clebsch--Gordon formula and Proposition \ref{elementaryJ}, we get \begin{equation}\Hom_{\lie g} (V(\lambda-\alpha_i), V(\mu_1)\otimes V(\mu_2))=\Hom_{\lie g} (V(\mu_1+\mu_2-\alpha_i), V(\mu_1)\otimes V(\mu_2))\ne 0.\end{equation}Using Lemma \ref{step1} this implies that \begin{equation} \label{disjoint}\Hom_{\lie g}(V(\lambda-\alpha_i), D(\ell,\lambda))\ne 0.\end{equation}
On the other hand since $\lambda(h_i)\le \ell$, we have  that the element $w_\lambda\in D(\ell,\lambda)$ satisfies the defining relation $(x_i^-\otimes t) w_\lambda=0$ and hence $$\bu(\lie g_J[t])w_\lambda\cong \bu(\lie g_J)w_\lambda\cong V_J(\lambda_J).$$ Using \eqref{demJ} we get
$$ \Hom_{\lie g}(V(\lambda-\alpha_i), D(\ell,\lambda))=0,$$ which contradicts \eqref{disjoint}. This proves the Lemma. \end{pf}

\subsection{} \begin{lem}\label{chooseJ} Suppose that $\nu_1,\nu_2\in P^+$ are  such that $$\supp \nu_1\cap\supp \nu_2=\emptyset.$$ There exists a connected subset  $J\subset I$ with  $\lie g_J$  isomorphic to $\lie{sl}_{r+1}$ for some $r\in\bn$ and  $$|J\cap\supp \nu_j|=\begin{cases} 1, \ \ \nu_j\ne 0,\\  0,\ \ \nu_j=0,\end{cases} ,\  \ \ 
j=1,2. $$
\end{lem}
\begin{pf}   If $\nu_1=\nu_2=0$, we take $J$ to be the empty set while if $\nu_1=0$ and $\nu_2\ne 0$ we take $J=\{i\}$ for some $i\in\supp\nu_2$. Assume now that $\nu_1$ and $\nu_2$ are non--zero. If $\lie g$ is of type $A_n$,
assume without loss of generality that $\supp \nu_2$ contains the maximal element in the union $\supp\nu_1\cup\supp \nu_2$.   Choose $i_1$ to be the maximal element in $\supp\nu_1$ and  $i_2\in \supp \nu_2$ minimal so that $i_2>i_1$. The minimal connected subset $J$  of $I$ containing $i_1$ and $i_2$ satisfies the conditions of the Lemma.

If $\lie g$ is of type $D$ or $E$ we let $i_0$ be the trivalent node and let $I_r$, $r=1,2,3$ be the three legs of the Dynkin diagram through $i_0$ and assume without loss of generality that $I_1=\{i_0,i_1\}$. Assume that $i_1\notin\supp\nu_2$.   Then, $$\nu_1' = \nu_1-\nu_1(h_{i_1})\omega_{i_1}\in P^+\ \ \supp \nu_1'\cap\supp \nu_2=\emptyset,\ \ \ i_1\notin\supp \nu_1'.$$  If $\nu_1'=0$ take $J$ to be the connected closure of $\{i_1,i_2\}$ for some $i_2\in\supp\nu_2$. If $\nu_1'\ne 0$, then
the connected closure of $\supp\nu_1'\cup\supp\nu_2$ is contained in $I_2\cup I_3$ and is of type $A$.
 Now, we can  use the result for $A$ to find $J\subset I\setminus \{i_1\}$ with the required properties for the pair $\nu_1'$, $\nu_2$.  But this set also has the desired properties for 
 the pair $\nu_1$, $\nu_2$ and the proof is complete.

\end{pf}

\subsection{} We return to the notation of Section \ref{convention}. Using Lemma \ref{suppintempty} we see that we can choose $J$ as in Lemma \ref{chooseJ} for the pair $\mu_1$, $\mu_2$. Let  $\theta_J\in R^+_J$ be  the highest root of $\lie g_J$ and notice that $\lambda_J=\lambda(h_{i_1})\omega_{i_1}+\lambda(h_{i_2})\omega_{i_2}$. If we assume in addition that $\lambda(h_i)<\ell$ for all $i\in I$, then we see that: $\lambda(h_\alpha)<\ell$ for all $\alpha\in R^+_J$ with $\alpha\ne \theta_J$
and $\lambda(h_{\theta_J})<2\ell$.

 Hence  the following  relations hold  in $D(\ell, \lambda)$  \begin{gather*}(x^-_{\alpha}\otimes t)w_\lambda= 0,\ \ \alpha\in R^+_J,
\ \ \alpha\ne \theta_J, \ \ (x^-_{\theta_J}\otimes t^2)w_\lambda=0,\\   (x^-_{\theta_J}\otimes t)^r=0,\ \ r >p=\max\{0,  \lambda(h_{\theta_J})-\ell\}.\end{gather*}  It is again a standard fact that the elements $(x^-_{\theta_J}\otimes t)^sw_\lambda$ are non--zero if $0\le s\le p$. Using the Poincare--Birkhoff--Witt theorem, one sees that $$\bu(\lie{g}_J[t])w_\lambda=\sum_{s=0}^p \bu(\lie g_J)(x^-_{\theta_J}\otimes t)^s w_\lambda.$$
 Moreover, a simple calculation  shows  that $(x^-_{\theta_J}\otimes t)^sw_\lambda$, $s\in\bz_+$  are $\lie n^+$--invariant vectors in $D(\ell,\lambda)$
and we have  $$\bu(\lie{g}_J[t])w_\lambda\cong_{\lie g_J}\bigoplus_{s=0}^p \bu(\lie g_J)(x^-_{\theta_J}\otimes t)^s w_\lambda\cong_{\lie g_J} \bigoplus_{s=0}^p V_J(\lambda_J-s\theta_J)^{m_s}.$$ Applying \eqref{demJ}, now gives \begin{equation}\label{homzwero1}\Hom_{\lie g}(V(\lambda-s\theta_J), D(\ell,\lambda)) =0,\ \ s > p.\end{equation}
On the other hand, it is well--known and in any case easily proved that
 \begin{equation*} \dim\Hom_{\lie g} (V(\mu_1+\mu_2-s\theta_J), V(\mu_1)\otimes V(\mu_2))\ne 0\ \  {\rm{if}}\ \ 0\le s\le \min\{\mu_1(h_{\theta_J}),\mu_2(h_{\theta_J})\}.\end{equation*}
Since $V(\mu_1)\otimes V(\mu_2)$ is isomorphic to a $\lie g$--submodule of $D(\ell,\lambda)$, it follows that
 \begin{equation}\label{tensoram} \dim\Hom_{\lie g} (V(\lambda-s\theta_J), D(\ell,\lambda))\ne 0\ \  {\rm{if}}\ \ 0\le s\le \min\{\mu_1(h_{\theta_J}),\mu_2(h_{\theta_J})\}.\end{equation}
Since $$ p=\max\{0,\lambda(h_{\theta_J})-\ell\}=\max\{0, \mu_1(h_{\theta_J})+\mu_2(h_{\theta_J})-\ell\} <\min\{\mu_1(h_{\theta_J}),\mu_2(h_{\theta_J})\},$$ we see that \eqref{tensoram} contradicts \eqref{homzwero1}. The proof of Proposition \ref{primefactor} is complete.

\section{Proof of Proposition \ref{key}}\label{keyproof}
\subsection{}  For $w\in W$ and $\lambda, \mu\in P^+$, we have $$wt_\mu(\ell\Lambda_0+w_0\lambda)= \ell\Lambda_0 +w(\ell\mu +w_0\lambda)+A\delta$$ for some $A\in\bz$. Hence, $wt_\mu(\ell\Lambda_0+w_0\lambda)\in\hat P^+$ iff $w\in W$ is such that $$w(\ell\mu+w_0\lambda)\in P^+\quad {\rm{and}} \quad \ \ \\\ w(\ell\mu+w_0\lambda)(h_\theta)\le\ell.$$  This shows that Proposition \ref{key} is an immediate consequence of the following,
\begin{lem}\label{key2} Given $(\ell, \lambda)\in \bn\times \lie h^*$ with $0\le\lambda(h_i)\le d_i\ell$  (equivalently that $0\le (\lambda,\alpha_i)\le\ell)$) for $1\le i\le n$, there exists $\mu\in L^+$ such that \begin{equation}\label{keyineq} |(\ell\mu-\lambda, \alpha)|\le \ell,\end{equation} for all $\alpha\in R^+$.
\end{lem}
The Lemma is proved in the rest of the section. The strategy for proving the Lemma is as follows. We give an inductive construction of $\mu$  in the case of $\lie g = C_n$ and use elementary results on root systems to deduce the existence of $\mu$ in the other classical cases.
In the case of $G_2$, we write down explicit solutions of $\mu$.  {\em From now on, we will assume that $(\ell,\lambda)$ are fixed and satisfy the conditions of the Lemma. We remind the reader that we are working with the form on $\lie h^*$ which has been normalized so that the square length of a long root is two.}

\subsection{Type $C$}\label{construction2} \begin{lem} Assume that $\lie g$ is of type $C_n$ and that $\alpha_n$ is the unique long simple  root. There exists  $\mu= 2\sum_{i=1}^{n-1}s_i\omega_{i}$ with $s_i\in\{0,1\}$ satisfying $|(\ell\mu-\lambda, \alpha)|\le \ell$ for all $\alpha\in R^+$.
\end{lem}
\begin{pf}
Any short root  $\alpha\in R$  is one half the difference of two long roots and hence it suffices to find $\mu$ such that $|(\ell\mu-\lambda, \alpha)|\le \ell$ holds for the long roots.

We proceed by induction on $n$, with induction beginning  at $n=1$ where we can take $\mu=0$.
For the inductive step assume  that the result is proved for the  $C_{n-1}$--subdiagram of $C_n$ defined by the  simple roots $\{\alpha_2,\cdots\alpha_n\}$ of $C_n$. Let  $\mu'=2\sum_{j=2}^{n-1}s_i\omega_i\in L^+$, with $s_i\in\{0,1\}$  such that $$|(\ell\mu'-\lambda, \alpha)|\le \ell,$$ for all roots $\alpha$ of $C_{n-1}$.
 The only additional long root in $C_n$ is the highest root $\theta$. Moreover,  $\theta-2\alpha_1$ is a root of $C_{n-1}$ and so
  we take   $$\mu=\begin{cases}\mu'\ \   \text{if } |(\lambda,\theta)-\ell(\mu',\theta-2\alpha_1)|\le \ell,\\
\; 2\omega_1+\mu',  \ \  {\rm{otherwise}}.
\end{cases}$$ A simple calculation completes the proof.
\end{pf}

\subsection{Type $A$} The diagram  subalgebra of $C_n$ generated by the root vectors $x_i^\pm$, $1\le i\le n-1$ is isomorphic to $A_{n-1}$ and the restriction of the fundamental weights $\omega_i$, $1\le i\le n-1$ of $C_n$ to $A_{n-1}$ gives a set fundamental weights for $A_{n-1}$.    There is one important thing to note here however. The restriction of the normalized form $(\ ,\ )$ of $C_n$ to the $A_{n-1}$ subdiagram is  one half of the normalized form on $A_{n-1}$.
This means that if $\lambda$ is any element in the real span of $\omega_i$, $1\le i\le n-1$ satisfying the conditions
of Lemma \ref{key2} of $A_{n-1}$ with respect to its normalized form, then the element $2\lambda$ regarded as an element of $C_n$ satisfies  $0\le (2\lambda,\alpha_i)\le \ell$ for all $1\le i\le n$ with respect to the normalized form on $C_n$. Hence we  can find $\mu=\sum_{i=1}^{n-1}s_i\omega_i$, with  $s_i\in\{0,1\}$ such that
$$|(2\lambda-2\ell\mu, \alpha)|\le \ell,$$ for all short  roots $\alpha$ of $C_n$ and hence for all roots of  $A_{n-1}$. This gives that $\mu$ satisfies \eqref{keyineq} for $\lambda$ with respect to the form on $A_{n-1}$ and the Lemma is established in this case.

\subsection{Type $D$}\label{dn} To prove the Lemma for $D_n$, we observe that the subset of short roots of $C_n$ form a root system of type $D_n$. Notice again that the restriction of the  normalized form on $C_n$ to $D_n$ is one half the normalized form of $D_n$.
 The simple system for $D_n$ is the set $\{\alpha_i: 1\le i\le n-1\}\cup\{\alpha_{n-1}+\alpha_n\}$ and  the set of fundamental weights  is $\{\omega_i: 1\le i\le n-2\}\cup\{ \omega_{n-1}-\frac{1}{2}\omega_n,\ \ \frac{1}{2}\omega_n\}$. In particular this means that if $\lambda$ is in the real span of the fundamental weights for $D_n$ satisfying the hypothesis of Lemma \ref{key2}, then,  either $2\lambda$ or $2\lambda\sigma$ (here $\sigma$ is the diagram automorphism of $D_n$ which switches the spin nodes and leaves the others fixed)   satisfy the conditions for $C_n$. Hence
we can choose a dominant integral weight  for $C_n$  of the form $2\mu$ where $\mu=\sum_{i=1}^{n-1}s_i\omega_i$, $s_i\in\{0,1\}$,  $1\le i\le n-1$ such that
\newline
$$|2(\ell\mu-\lambda), {\alpha})|\leq \ell \ \  ({\rm{resp.}}\ \ |2(\ell\mu-\lambda\sigma), {\alpha})|\leq \ell)$$ for all short roots $\alpha$ of $C_n$, i.e., for all roots of $D_n$. Since $\mu$ and $\mu\sigma$ are dominant integral weights of $D_n$, Lemma \ref{key2} follows for the element $\lambda$ with $\mu$ or $\mu\sigma$ and the normalized form of $D_n$, according as  $2\lambda$ or $2\lambda\sigma$ is dominant for $C_n$. {\em We remark here that the element $\mu$  when regarded as an element of $D_n$ is such that
 it is either not supported on the spin nodes or it is supported on both spin nodes.  This is because either $s_{n-1}=0$ in which case it is not supported on the spin nodes or $s_{n-1}=1$ and we have $$\mu = \sum_{i=1}^{n-2}s_i\omega_i + (\omega_{n-1}-\frac{1}{2}\omega_n) + \frac{1}{2}\omega_n$$}

\subsection{Type $B$}  To prove the result for $B_n$ we first observe that it is enough to prove that there exists $\mu\in L^+$ such that  \eqref{keyineq} is satisfied for the long roots. This is because any short  root is half the difference of two long roots. Recall that $B_n$ can be regarded as a subalgebra of $D_{n+1}$ by folding: namely it is the fixed points of the automorphism $\sigma$ which interchanges the spin nodes and leaves the others fixed. If $\alpha_i$, $1\le i\le n+1$ are the simple roots of $D_{n+1}$, then the simple roots of $B_n$ are $\alpha_i$, $1\le i\le n-1$ and $\frac 12 (\alpha_n+\alpha_{n+1})$. It is easily seen that any long root of $B_n$ is a root of $D_{n+1}$.

The restriction of the normalized form of $D_{n+1}$ to $B_n$ is the normalized form of $B_n$. The set of dominant integral weights for $B_n$ is $\omega_i$, $1\le i\le n-1$, and $\frac 12(\omega_n+\omega_{n+1})$. Given $\lambda=\sum_{i=1}^{n-1}r_i\omega_i +r_n\frac 12(\omega_n+\omega_{n+1})$, one sees that if $\lambda$ satisfies the conditions of Lemma \ref{keyineq} for $B_n$, then  we have that $r_n\le 2\ell$ and hence
 $\lambda$ also satisfies the conditions  for $D_{n+1}$. Choose $\mu=\sum_{i=1}^{n+1}s_i\omega_i$ as in Section \ref{dn} such $ s_i\in\{0,1\}$ satisfies \eqref{keyineq} for $D_{n+1}$. Since either $s_n=s_{n+1}=0$ or $s_n=s_{n+1}=1$, we see that $\mu$ is in the lattice $L^+$ for $B_n$ and hence Lemma \ref{key2} follows for $B_n$.

 \subsection{Type $G_2$}\label{G_2}  If $\lie g$ is of type $G_2$, we assume that $\alpha_2$ is the simple short root. We note that it is enough to prove that there exists a $\mu\in L^+$, which satisfies \eqref{keyineq} only on long roots.  This is because any non-simple short root can be written as either a half or a third of the sum of two long roots.  Next, we observe that we have, $$(\omega_1,\alpha_1)=1,\ \ (\omega_2,\alpha_2)=1/3.$$   Let $\mu$ be the following weight in $L^+$, \newline
$$\mu
 =\begin{cases}
0, & \text{if }(\lambda,2\alpha_1+3\alpha_2)\leq\ell \\
\omega_1, & \text{if }\ell<(\lambda, 2\alpha_1+3\alpha_2)\leq3\ell \textrm{ and } (\lambda,\alpha_1+3\alpha_2)\leq 2\ell \\
3\omega_2, & \text{if }2\ell<(\lambda,2\alpha_1+3\alpha_2)\leq4\ell \textrm{ and } (\lambda,\alpha_1+3\alpha_2)> 2\ell \\
\omega_1+3\omega_2, & \text{if }4\ell<(\lambda,2\alpha_1+3\alpha_2)\leq5\ell \\
\end{cases}$$
where we note that the last condition $4\ell < (\lambda, 2\alpha_1+3\alpha_2)$ implies that $(\lambda, \alpha_1+3\alpha_2)>3\ell$.  Therefore, one can check easily that the condition $|(\ell\mu - \lambda, \alpha)|\leq \ell$ is satisfied for all positive long roots, and hence all positive roots.

\subsection{The case of $E$ and $F_4$}  It is clear that it suffices to prove Proposition \ref{key} for $E_8$ and $F_4$. The methods of this section do not appear to generalize to these cases. However, it is possible to check using mathematica  that Proposition \ref{key}  is true for $\ell$ at least five.
In the tables in the appendix, we associate to the ordered pair $(a_1,\cdots,a_n)$ the weight $\nu = \sum_{i=1}^na_i\omega_i$.  For $\ell=2$, we provide one solution  for every  $\lambda$  with $\lambda(h_i)\le 1$ for all $1\le i\le n$.

\vfill\eject

\appendix

\section{Mathematica output:  $F_4$ and $\ell =2$}

\definecolor{lightblue}{RGB}{220,220,255}
\begin{table}[!h]
    \centering
    \begin{tabular}{||>{\columncolor[gray]{0.9}}c|c||>{\columncolor[gray]{0.9}}c|c||}
        \hline\hline
      \rowcolor{lightblue} $\lambda$&$\mu$&$\lambda$&$\mu$\\\hline\hline
        (0,0,0,0)&(0,0,0,0)&(1,0,0,0)&(0,0,0,0)\\\hline
        (0,0,0,1)&(0,0,0,0)&(1,0,0,1)&(1,0,0,0)\\\hline
        (0,0,0,2)&(0,0,0,0)&(1,0,0,2)&(0,0,0,2)\\\hline
        (0,0,0,3)&(0,0,0,2)&(1,0,0,3)&(0,0,0,2)\\\hline
        (0,0,1,0)&(0,0,0,0)&(1,0,1,0)&(1,0,0,0)\\\hline
        (0,0,1,1)&(0,0,0,2)&(1,0,1,1)&(0,0,0,2)\\\hline
        (0,0,1,2)&(0,0,0,2)&(1,0,1,2)&(0,0,0,2)\\\hline
        (0,0,1,3)&(0,0,0,2)&(1,0,1,3)&(1,0,0,2)\\\hline
        (0,0,2,0)&(0,0,0,2)&(1,0,2,0)&(0,1,0,0)\\\hline
        (0,0,2,1)&(0,0,0,2)&(1,0,2,1)&(0,0,2,0)\\\hline
        (0,0,2,2)&(0,0,0,2)&(1,0,2,2)&(0,0,2,0)\\\hline
        (0,0,2,3)&(0,0,2,0)&(1,0,2,3)&(0,0,2,0)\\\hline
        (0,0,3,0)&(0,0,2,0)&(1,0,3,0)&(0,0,2,0)\\\hline
        (0,0,3,1)&(0,0,2,0)&(1,0,3,1)&(0,0,2,0)\\\hline
        (0,0,3,2)&(0,0,2,0)&(1,0,3,2)&(0,0,2,0)\\\hline
        (0,0,3,3)&(0,0,2,0)&(1,0,3,3)&(0,0,2,2)\\\hline
         (0,1,0,0)&(1,0,0,0)&(1,1,0,0)&(0,1,0,0)\\\hline
        (0,1,0,1)&(0,0,0,2)&(1,1,0,1)&(0,1,0,0)\\\hline
        (0,1,0,2)&(0,0,0,2)&(1,1,0,2)&(0,1,0,0)\\\hline
        (0,1,0,3)&(0,0,0,2)&(1,1,0,3)&(1,0,0,2)\\\hline
        (0,1,1,0)&(0,1,0,0)&(1,1,1,0)&(0,1,0,0)\\\hline
        (0,1,1,1)&(0,1,0,0)&(1,1,1,1)&(0,0,2,0)\\\hline
        (0,1,1,2)&(0,0,2,0)&(1,1,1,2)&(0,0,2,0)\\\hline
        (0,1,1,3)&(0,0,2,0)&(1,1,1,3)&(0,1,0,2)\\\hline
        (0,1,2,0)&(0,0,2,0)&(1,1,2,0)&(0,0,2,0)\\\hline
        (0,1,2,1)&(0,0,2,0)&(1,1,2,1)&(0,0,2,0)\\\hline
        (0,1,2,2)&(0,0,2,0)&(1,1,2,2)&(0,1,0,2)\\\hline
        (0,1,2,3)&(0,1,0,2)&(1,1,2,3)&(0,0,2,2)\\\hline
        (0,1,3,0)&(0,0,2,0)&(1,1,3,0)&(1,0,2,0)\\\hline
        (0,1,3,1)&(0,0,2,0)&(1,1,3,1)&(1,0,2,0)\\\hline
        (0,1,3,2)&(0,0,2,2)&(1,1,3,2)&(0,0,2,2)\\\hline
        (0,1,3,3)&(0,0,2,2)&(1,1,3,3)&(0,0,2,2)\\\hline

    \end{tabular}
\end{table}

\vfill\eject

\section{Mathematica output:  $E_8$ and $\ell=2$}

\definecolor{lightblue}{RGB}{220,220,255}
\begin{table}[h!]
    \centering
    \begin{tabular}{||>{\columncolor[gray]{0.9}}c|c||>{\columncolor[gray]{0.9}}c|c||}
        \hline\hline
      \rowcolor{lightblue} $\lambda$&$\mu$&$\lambda$&$\mu$\\\hline\hline
        (0,0,0,0,0,0,0,0)&(0,0,0,0,0,0,0,0)&(0,0,0,1,0,1,0,0)&(0,0,0,0,1,0,0,0)\\\hline
        (1,0,0,0,0,0,0,0)&(1,0,0,0,0,0,0,0)&(1,0,0,1,0,1,0,0)&(0,0,0,1,0,0,0,0)\\\hline
        (0,1,0,0,0,0,0,0)&(0,0,0,0,0,0,0,1)&(0,1,0,1,0,1,0,0)&(0,0,0,1,0,0,0,0)\\\hline
        (1,1,0,0,0,0,0,0)&(0,1,0,0,0,0,0,0)&(1,1,0,1,0,1,0,0)&(0,0,0,1,0,0,0,1)\\\hline
        (0,0,1,0,0,0,0,0)&(1,0,0,0,0,0,0,0)&(0,0,1,1,0,1,0,0)&(0,0,0,1,0,0,0,0)\\\hline
        (1,0,1,0,0,0,0,0)&(0,0,1,0,0,0,0,0)&(1,0,1,1,0,1,0,0)&(0,0,1,0,0,1,0,0)\\\hline
        (0,1,1,0,0,0,0,0)&(0,0,1,0,0,0,0,0)&(0,1,1,1,0,1,0,0)&(0,0,0,1,0,0,1,0)\\\hline
        (1,1,1,0,0,0,0,0)&(0,0,1,0,0,0,0,0)&(1,1,1,1,0,1,0,0)&(0,0,0,1,0,1,0,0)\\\hline
        (0,0,0,1,0,0,0,0)&(0,1,0,0,0,0,0,0)&(0,0,0,0,1,1,0,0)&(0,0,0,0,0,1,0,0)\\\hline
        (1,0,0,1,0,0,0,0)&(0,0,1,0,0,0,0,0)&(1,0,0,0,1,1,0,0)&(0,0,0,0,1,0,0,0)\\\hline
        (0,1,0,1,0,0,0,0)&(0,0,0,0,1,0,0,0)&(0,1,0,0,1,1,0,0)&(0,0,0,0,1,0,0,0)\\\hline
        (1,1,0,1,0,0,0,0)&(0,0,0,1,0,0,0,0)&(1,1,0,0,1,1,0,0)&(0,1,0,0,0,1,0,0)\\\hline
        (0,0,1,1,0,0,0,0)&(0,0,0,0,1,0,0,0)&(0,0,1,0,1,1,0,0)&(0,0,0,0,1,0,0,1)\\\hline
        (1,0,1,1,0,0,0,0)&(0,0,0,1,0,0,0,0)&(1,0,1,0,1,1,0,0)&(0,0,1,0,0,1,0,0)\\\hline
        (0,1,1,1,0,0,0,0)&(0,0,0,1,0,0,0,0)&(0,1,1,0,1,1,0,0)&(0,0,1,0,0,1,0,0)\\\hline
        (1,1,1,1,0,0,0,0)&(0,1,1,0,0,0,0,0)&(1,1,1,0,1,1,0,0)&(0,0,1,0,1,0,0,0)\\\hline
        (0,0,0,0,1,0,0,0)&(0,0,0,0,0,0,1,0)&(0,0,0,1,1,1,0,0)&(0,0,0,0,1,0,1,0)\\\hline
        (1,0,0,0,1,0,0,0)&(0,0,0,0,0,1,0,0)&(1,0,0,1,1,1,0,0)&(0,0,0,0,1,1,0,0)\\\hline
        (0,1,0,0,1,0,0,0)&(0,0,0,0,0,1,0,0)&(0,1,0,1,1,1,0,0)&(0,0,0,0,1,1,0,0)\\\hline
        (1,1,0,0,1,0,0,0)&(0,0,0,0,1,0,0,0)&(1,1,0,1,1,1,0,0)&(0,0,0,1,0,1,0,0)\\\hline
        (0,0,1,0,1,0,0,0)&(0,0,0,0,1,0,0,0)&(0,0,1,1,1,1,0,0)&(0,0,0,1,0,1,0,0)\\\hline
        (1,0,1,0,1,0,0,0)&(0,0,0,0,1,0,0,0)&(1,0,1,1,1,1,0,0)&(0,0,0,1,0,1,0,0)\\\hline
        (0,1,1,0,1,0,0,0)&(0,0,0,1,0,0,0,0)&(0,1,1,1,1,1,0,0)&(0,0,0,1,1,0,0,0)\\\hline
        (1,1,1,0,1,0,0,0)&(0,1,1,0,0,0,0,0)&(1,1,1,1,1,1,0,0)&(0,1,1,0,1,0,0,0)\\\hline
        (0,0,0,1,1,0,0,0)&(0,0,0,0,1,0,0,0)&(0,0,0,0,0,0,1,0)&(0,0,0,0,0,0,0,1)\\\hline
        (1,0,0,1,1,0,0,0)&(0,0,0,1,0,0,0,0)&(1,0,0,0,0,0,1,0)&(0,0,0,0,0,0,1,0)\\\hline
        (0,1,0,1,1,0,0,0)&(0,0,0,1,0,0,0,0)&(0,1,0,0,0,0,1,0)&(0,0,0,0,0,0,1,0)\\\hline
        (1,1,0,1,1,0,0,0)&(0,1,0,0,1,0,0,0)&(1,1,0,0,0,0,1,0)&(0,0,0,0,0,1,0,0)\\\hline
        (0,0,1,1,1,0,0,0)&(0,0,0,1,0,0,0,1)&(0,0,1,0,0,0,1,0)&(0,0,0,0,0,0,1,0)\\\hline
        (1,0,1,1,1,0,0,0)&(0,0,1,0,1,0,0,0)&(1,0,1,0,0,0,1,0)&(0,0,1,0,0,0,0,0)\\\hline
        (0,1,1,1,1,0,0,0)&(0,0,1,0,1,0,0,0)&(0,1,1,0,0,0,1,0)&(0,0,0,0,1,0,0,0)\\\hline
        (1,1,1,1,1,0,0,0)&(0,0,1,1,0,0,0,0)&(1,1,1,0,0,0,1,0)&(0,0,0,1,0,0,0,0)\\\hline
        (0,0,0,0,0,1,0,0)&(0,0,0,0,0,0,0,1)&(0,0,0,1,0,0,1,0)&(0,0,0,0,0,1,0,0)\\\hline
        (1,0,0,0,0,1,0,0)&(0,0,0,0,0,0,1,0)&(1,0,0,1,0,0,1,0)&(0,0,0,0,1,0,0,0)\\\hline
        (0,1,0,0,0,1,0,0)&(0,0,0,0,0,1,0,0)&(0,1,0,1,0,0,1,0)&(0,0,0,1,0,0,0,0)\\\hline
        (1,1,0,0,0,1,0,0)&(0,0,0,0,0,1,0,0)&(1,1,0,1,0,0,1,0)&(0,0,0,1,0,0,0,0)\\\hline
        (0,0,1,0,0,1,0,0)&(0,0,0,0,0,1,0,0)&(0,0,1,1,0,0,1,0)&(0,0,0,1,0,0,0,0)\\\hline
        (1,0,1,0,0,1,0,0)&(0,0,0,0,1,0,0,0)&(1,0,1,1,0,0,1,0)&(0,0,0,1,0,0,1,0)\\\hline
        (0,1,1,0,0,1,0,0)&(0,0,0,0,1,0,0,0)&(0,1,1,1,0,0,1,0)&(0,0,0,1,0,0,0,1)\\\hline

    \end{tabular}
\end{table}

\begin{table}[h!]
    \centering
    \begin{tabular}{||>{\columncolor[gray]{0.9}}c|c||>{\columncolor[gray]{0.9}}c|c||}
        \hline\hline
      \rowcolor{lightblue} $\lambda$&$\mu$&$\lambda$&$\mu$\\\hline\hline
 (1,1,1,0,0,1,0,0)&(0,0,0,1,0,0,0,0)&(1,1,1,1,0,0,1,0)&(0,0,0,1,0,0,1,0)\\\hline
           (0,0,0,0,1,0,1,0)&(0,0,0,0,0,1,0,0)&(0,0,0,1,1,1,1,0)&(0,0,0,0,1,1,0,0)\\\hline
        (1,0,0,0,1,0,1,0)&(0,0,0,0,1,0,0,0)&(1,0,0,1,1,1,1,0)&(0,0,0,1,0,1,0,0)\\\hline
        (0,1,0,0,1,0,1,0)&(0,0,0,0,1,0,0,0)&(0,1,0,1,1,1,1,0)&(0,0,0,1,0,1,0,0)\\\hline
        (1,1,0,0,1,0,1,0)&(0,0,0,0,1,0,0,1)&(1,1,0,1,1,1,1,0)&(0,0,0,1,0,1,0,1)\\\hline
        (0,0,1,0,1,0,1,0)&(0,0,0,0,1,0,0,0)&(0,0,1,1,1,1,1,0)&(0,0,0,1,0,1,0,0)\\\hline
        (1,0,1,0,1,0,1,0)&(0,0,1,0,0,0,1,0)&(1,0,1,1,1,1,1,0)&(0,0,1,0,1,0,1,0)\\\hline
        (0,1,1,0,1,0,1,0)&(0,0,0,0,1,0,1,0)&(0,1,1,1,1,1,1,0)&(0,0,0,1,0,1,1,0)\\\hline
        (1,1,1,0,1,0,1,0)&(0,0,0,1,0,0,1,0)&(1,1,1,1,1,1,1,0)&(0,0,0,1,1,0,1,0)\\\hline
        (0,0,0,1,1,0,1,0)&(0,0,0,0,1,0,0,1)&(0,0,0,0,0,0,0,1)&(0,0,0,0,0,0,0,0)\\\hline
        (1,0,0,1,1,0,1,0)&(0,0,0,0,1,0,1,0)&(1,0,0,0,0,0,0,1)&(0,0,0,0,0,0,0,1)\\\hline
        (0,1,0,1,1,0,1,0)&(0,0,0,1,0,0,1,0)&(0,1,0,0,0,0,0,1)&(0,0,0,0,0,0,0,1)\\\hline
        (1,1,0,1,1,0,1,0)&(0,0,0,1,0,0,1,0)&(1,1,0,0,0,0,0,1)&(0,1,0,0,0,0,0,0)\\\hline
        (0,0,1,1,1,0,1,0)&(0,0,0,1,0,0,1,0)&(0,0,1,0,0,0,0,1)&(0,0,0,0,0,0,1,0)\\\hline
        (1,0,1,1,1,0,1,0)&(0,0,0,1,0,1,0,0)&(1,0,1,0,0,0,0,1)&(0,0,1,0,0,0,0,0)\\\hline
        (0,1,1,1,1,0,1,0)&(0,0,0,1,0,1,0,0)&(0,0,0,1,0,0,0,1)&(0,0,0,0,0,1,0,0)\\\hline
        (1,1,1,1,1,0,1,0)&(0,0,0,1,1,0,0,0)&(1,0,0,1,0,0,0,1)&(0,0,0,0,1,0,0,0)\\\hline
        (0,0,0,0,0,1,1,0)&(0,0,0,0,0,0,1,0)&(0,1,0,1,0,0,0,1)&(0,0,0,0,1,0,0,0)\\\hline
        (1,0,0,0,0,1,1,0)&(0,0,0,0,0,1,0,0)&(1,1,0,1,0,0,0,1)&(0,0,0,1,0,0,0,0)\\\hline
        (0,1,0,0,0,1,1,0)&(0,0,0,0,0,1,0,0)&(0,0,1,1,0,0,0,1)&(0,0,0,1,0,0,0,0)\\\hline
        (1,1,0,0,0,1,1,0)&(0,1,0,0,0,0,1,0)&(1,0,1,1,0,0,0,1)&(0,0,0,1,0,0,0,0)\\\hline
        (0,0,1,0,0,1,1,0)&(0,0,0,0,0,1,0,1)&(0,1,1,1,0,0,0,1)&(0,0,0,1,0,0,0,1)\\\hline
        (1,0,1,0,0,1,1,0)&(0,0,1,0,0,0,1,0)&(1,1,1,1,0,0,0,1)&(0,1,1,0,0,0,0,1)\\\hline
        (0,1,1,0,0,1,1,0)&(0,0,1,0,0,0,1,0)&(0,0,0,0,1,0,0,1)&(0,0,0,0,0,0,1,0)\\\hline
        (1,1,1,0,0,1,1,0)&(0,0,1,0,0,1,0,0)&(1,0,0,0,1,0,0,1)&(0,0,0,0,0,1,0,0)\\\hline
        (0,0,0,1,0,1,1,0)&(0,0,0,0,0,1,1,0)&(0,1,0,0,1,0,0,1)&(0,0,0,0,1,0,0,0)\\\hline
        (1,0,0,1,0,1,1,0)&(0,0,0,0,1,0,1,0)&(1,1,0,0,1,0,0,1)&(0,0,0,0,1,0,0,0)\\\hline
        (0,1,0,1,0,1,1,0)&(0,0,0,0,1,0,1,0)&(0,0,1,0,1,0,0,1)&(0,0,0,0,1,0,0,0)\\\hline
        (1,1,0,1,0,1,1,0)&(0,0,0,1,0,0,1,0)&(1,0,1,0,1,0,0,1)&(0,0,0,0,1,0,0,1)\\\hline
        (0,0,1,1,0,1,1,0)&(0,0,0,1,0,0,1,0)&(0,1,1,0,1,0,0,1)&(0,0,0,0,1,0,0,1)\\\hline
        (1,0,1,1,0,1,1,0)&(0,0,0,1,0,0,1,0)&(1,1,1,0,1,0,0,1)&(0,0,0,1,0,0,0,1)\\\hline
        (0,1,1,1,0,1,1,0)&(0,0,0,1,0,1,0,0)&(0,0,0,1,1,0,0,1)&(0,0,0,0,1,0,0,1)\\\hline
        (1,1,1,1,0,1,1,0)&(0,1,1,0,0,1,0,0)&(1,0,0,1,1,0,0,1)&(0,0,0,1,0,0,0,1)\\\hline
        (0,0,0,0,1,1,1,0)&(0,0,0,0,0,1,0,1)&(0,1,0,1,1,0,0,1)&(0,0,0,1,0,0,0,1)\\\hline
        (1,0,0,0,1,1,1,0)&(0,0,0,0,0,1,1,0)&(1,1,0,1,1,0,0,1)&(0,0,0,1,0,0,1,0)\\\hline
        (0,1,0,0,1,1,1,0)&(0,0,0,0,1,0,1,0)&(0,0,1,1,1,0,0,1)&(0,0,0,1,0,0,0,1)\\\hline
        (1,1,0,0,1,1,1,0)&(0,0,0,0,1,0,1,0)&(1,0,1,1,1,0,0,1)&(0,0,1,0,1,0,0,0)\\\hline
        (0,0,1,0,1,1,1,0)&(0,0,0,0,1,0,1,0)&(0,1,1,1,1,0,0,1)&(0,0,0,1,0,1,0,0)\\\hline
        (1,0,1,0,1,1,1,0)&(0,0,0,0,1,1,0,0)&(1,1,1,1,1,0,0,1)&(0,0,0,1,1,0,0,0)\\\hline
        (0,1,1,0,1,1,1,0)&(0,0,0,0,1,1,0,0)&(0,0,0,0,0,1,0,1)&(0,0,0,0,0,0,1,0)\\\hline
        (1,1,1,0,1,1,1,0)&(0,0,0,1,0,1,0,0)&(1,0,0,0,0,1,0,1)&(0,0,0,0,0,1,0,0)\\\hline
        
    \end{tabular}
\end{table}

\begin{table}[h!]
    \centering
    \begin{tabular}{||>{\columncolor[gray]{0.9}}c|c||>{\columncolor[gray]{0.9}}c|c||}
        \hline\hline
      \rowcolor{lightblue} $\lambda$&$\mu$&$\lambda$&$\mu$\\\hline\hline
        (0,1,0,0,0,1,0,1)&(0,0,0,0,0,1,0,0)&(0,0,0,1,0,0,1,1)&(0,0,0,0,0,1,0,1)\\\hline
        (0,1,0,0,0,1,0,1)&(0,0,0,0,0,1,0,0)&(1,0,0,1,0,0,1,1)&(0,0,0,0,1,0,0,1)\\\hline
        (1,1,0,0,0,1,0,1)&(0,0,0,0,0,1,0,1)&(0,1,0,1,0,0,1,1)&(0,0,0,0,1,0,0,1)\\\hline
        (0,0,1,0,0,1,0,1)&(0,0,0,0,0,1,0,0)&(1,1,0,1,0,0,1,1)&(0,0,0,1,0,0,0,1)\\\hline
        (1,0,1,0,0,1,0,1)&(0,0,1,0,0,0,0,1)&(0,0,1,1,0,0,1,1)&(0,0,0,1,0,0,0,1)\\\hline
        (0,1,1,0,0,1,0,1)&(0,0,0,0,1,0,0,1)&(1,0,1,1,0,0,1,1)&(0,0,0,1,0,0,0,1)\\\hline
        (1,1,1,0,0,1,0,1)&(0,0,0,1,0,0,0,1)&(0,1,1,1,0,0,1,1)&(0,0,0,1,0,0,1,0)\\\hline
        (0,0,0,1,0,1,0,1)&(0,0,0,0,0,1,0,1)&(1,1,1,1,0,0,1,1)&(0,1,1,0,0,0,1,0)\\\hline
        (0,0,0,1,0,1,0,1)&(0,0,0,0,0,1,0,1)&(0,0,0,0,1,0,1,1)&(0,0,0,0,0,0,1,1)\\\hline
        (1,0,0,1,0,1,0,1)&(0,0,0,0,1,0,0,1)&(1,0,0,0,1,0,1,1)&(0,0,0,0,0,1,0,1)\\\hline
        (0,1,0,1,0,1,0,1)&(0,0,0,1,0,0,0,1)&(0,1,0,0,1,0,1,1)&(0,0,0,0,1,0,0,1)\\\hline
        (1,1,0,1,0,1,0,1)&(0,0,0,1,0,0,0,1)&(1,1,0,0,1,0,1,1)&(0,0,0,0,1,0,0,1)\\\hline
        (0,0,1,1,0,1,0,1)&(0,0,0,1,0,0,0,1)&(0,0,1,0,1,0,1,1)&(0,0,0,0,1,0,0,1)\\\hline
        (1,0,1,1,0,1,0,1)&(0,0,0,1,0,0,1,0)&(1,0,1,0,1,0,1,1)&(0,0,0,0,1,0,1,0)\\\hline
        (0,1,1,1,0,1,0,1)&(0,0,0,1,0,0,1,0)&(0,1,1,0,1,0,1,1)&(0,0,0,0,1,0,1,0)\\\hline
        (1,1,1,1,0,1,0,1)&(0,0,0,1,0,1,0,0)&(1,1,1,0,1,0,1,1)&(0,0,0,1,0,0,1,0)\\\hline
        (0,0,0,0,1,1,0,1)&(0,0,0,0,0,1,0,1)&(0,0,0,1,1,0,1,1)&(0,0,0,0,1,0,1,0)\\\hline
        (1,0,0,0,1,1,0,1)&(0,0,0,0,1,0,0,1)&(1,0,0,1,1,0,1,1)&(0,0,0,1,0,0,1,0)\\\hline
        (0,1,0,0,1,1,0,1)&(0,0,0,0,1,0,0,1)&(0,1,0,1,1,0,1,1)&(0,0,0,1,0,0,1,0)\\\hline
        (1,1,0,0,1,1,0,1)&(0,0,0,0,1,0,1,0)&(1,1,0,1,1,0,1,1)&(0,0,0,1,0,0,1,1)\\\hline
        (0,0,1,0,1,1,0,1)&(0,0,0,0,1,0,0,1)&(0,0,1,1,1,0,1,1)&(0,0,0,1,0,0,1,0)\\\hline
        (1,0,1,0,1,1,0,1)&(0,0,1,0,0,1,0,0)&(1,0,1,1,1,0,1,1)&(0,0,1,0,1,0,0,1)\\\hline
        (0,1,1,0,1,1,0,1)&(0,0,0,0,1,1,0,0)&(0,1,1,1,1,0,1,1)&(0,0,0,1,0,1,0,1)\\\hline
        (1,1,1,0,1,1,0,1)&(0,0,0,1,0,1,0,0)&(1,1,1,1,1,0,1,1)&(0,0,0,1,1,0,0,1)\\\hline
        (0,0,0,1,1,1,0,1)&(0,0,0,0,1,0,1,0)&(0,0,0,0,0,1,1,1)&(0,0,0,0,0,0,1,1)\\\hline
        (1,0,0,1,1,1,0,1)&(0,0,0,0,1,1,0,0)&(1,0,0,0,0,1,1,1)&(0,0,0,0,0,1,0,1)\\\hline
        (0,1,0,1,1,1,0,1)&(0,0,0,1,0,1,0,0)&(0,1,0,0,0,1,1,1)&(0,0,0,0,0,1,0,1)\\\hline
        (1,1,0,1,1,1,0,1)&(0,0,0,1,0,1,0,0)&(1,1,0,0,0,1,1,1)&(0,0,0,0,0,1,1,0)\\\hline
        (0,0,1,1,1,1,0,1)&(0,0,0,1,0,1,0,0)&(0,0,1,0,0,1,1,1)&(0,0,0,0,0,1,0,1)\\\hline
        (1,0,1,1,1,1,0,1)&(0,0,0,1,0,1,0,1)&(1,0,1,0,0,1,1,1)&(0,0,1,0,0,0,1,0)\\\hline
        (0,1,1,1,1,1,0,1)&(0,0,0,1,0,1,0,1)&(0,1,1,0,0,1,1,1)&(0,0,0,0,1,0,1,0)\\\hline
        (1,1,1,1,1,1,0,1)&(0,0,0,1,1,0,0,1)&(1,1,1,0,0,1,1,1)&(0,0,0,1,0,0,1,0)\\\hline
        (0,0,0,0,0,0,1,1)&(0,0,0,0,0,0,0,1)&(0,0,0,1,0,1,1,1)&(0,0,0,0,0,1,1,0)\\\hline
        (1,0,0,0,0,0,1,1)&(0,0,0,0,0,0,1,0)&(1,0,0,1,0,1,1,1)&(0,0,0,0,1,0,1,0)\\\hline
        (0,1,0,0,0,0,1,1)&(0,0,0,0,0,0,1,0)&(0,1,0,1,0,1,1,1)&(0,0,0,1,0,0,1,0)\\\hline
        (1,1,0,0,0,0,1,1)&(0,1,0,0,0,0,0,1)&(1,1,0,1,0,1,1,1)&(0,0,0,1,0,0,1,0)\\\hline
        (0,0,1,0,0,0,1,1)&(0,0,0,0,0,0,1,1)&(0,0,1,1,0,1,1,1)&(0,0,0,1,0,0,1,0)\\\hline
        (1,0,1,0,0,0,1,1)&(0,0,1,0,0,0,0,1)&(1,0,1,1,0,1,1,1)&(0,0,0,1,0,0,1,1)\\\hline
        (0,1,1,0,0,0,1,1)&(0,0,1,0,0,0,0,1)&(0,1,1,1,0,1,1,1)&(0,0,0,1,0,0,1,1)\\\hline
        (1,1,1,0,0,0,1,1)&(0,0,1,0,0,0,1,0)&(1,1,1,1,0,1,1,1)&(0,0,0,1,0,1,0,1)\\\hline
        
    \end{tabular}
\end{table}

\begin{table}[h!]
    \centering
    \begin{tabular}{||>{\columncolor[gray]{0.9}}c|c||>{\columncolor[gray]{0.9}}c|c||}
        \hline\hline
      \rowcolor{lightblue} $\lambda$&$\mu$&$\lambda$&$\mu$\\\hline\hline
        (0,0,0,0,1,1,1,1)&(0,0,0,0,0,1,1,0)&(0,0,0,1,1,1,1,1)&(0,0,0,0,1,0,1,1)\\\hline
        (1,0,0,0,1,1,1,1)&(0,0,0,0,1,0,1,0)&(1,0,0,1,1,1,1,1)&(0,0,0,0,1,1,0,1)\\\hline
        (0,1,0,0,1,1,1,1)&(0,0,0,0,1,0,1,0)&(0,1,0,1,1,1,1,1)&(0,0,0,1,0,1,0,1)\\\hline
        (1,1,0,0,1,1,1,1)&(0,0,0,0,1,0,1,1)&(1,1,0,1,1,1,1,1)&(0,0,0,1,0,1,0,1)\\\hline
        (0,0,1,0,1,1,1,1)&(0,0,0,0,1,0,1,0)&(0,0,1,1,1,1,1,1)&(0,0,0,1,0,1,0,1)\\\hline
        (1,0,1,0,1,1,1,1)&(0,0,1,0,0,1,0,1)&(1,0,1,1,1,1,1,1)&(0,0,0,1,0,1,1,0)\\\hline
        (0,1,1,0,1,1,1,1)&(0,0,0,0,1,1,0,1)&(0,1,1,1,1,1,1,1)&(0,0,0,1,0,1,1,0)\\\hline
        (1,1,1,0,1,1,1,1)&(0,0,0,1,0,1,0,1)&(1,1,1,1,1,1,1,1)&(0,0,0,1,1,0,1,0)\\\hline

    \end{tabular}
\end{table}

\vfill\eject

\end{document}